\let\iflabels\iffalse
\let\iftxfonts\iftrue
\let\iflitnum\iftrue
\let\ifsrcltx\iftrue
\let\ifindexindication\iffalse

\documentclass[11pt]{amsart}

\usepackage{amssymb,amsmath}
\usepackage[all]{xy}
\usepackage{graphicx}

\newcommand\CC{{\mathbb{C}}}

\newcommand\QQ{{\mathbb{Q}}}
\newcommand\RR{{\mathbb{R}}}
\newcommand\ZZ{{\mathbb{Z}}}

\newcommand\al{\alpha}
\newcommand\Bt{{\mathsf{B}}}
\newcommand\bt{\beta}
\newcommand\Gm{\Gamma}
\newcommand\gm{\gamma}

\newcommand\dt{\delta}
\newcommand\e{\varepsilon}

\renewcommand\th{\vartheta}
\newcommand\kp{\kappa}
\newcommand\s{\sigma}
\newcommand\Ph{\Phi}
\newcommand\ph{\varphi}
\newcommand\ch{\chi}
\newcommand\Ps{\Psi}
\newcommand\ps{\psi}
\newcommand\Om{\Omega}
\newcommand\om{\omega}

\newcommand\im{{\mathrm{Im}\,}}
\newcommand\cp{{\mathrm{cp}}}
\newcommand\SL{{\mathrm{SL}}}
\newcommand\PSL{{\mathrm{PSL}}}
\newcommand\hp{{\mathfrak H}}
\newcommand\uhp{{\mathfrak H}}
\newcommand\lhp{{\mathfrak H}^-}
\newcommand\cf[1]{S_{\!#1}}
\newcommand\oh{\mathrm{O}}
\newcommand\dsv[2]{{}^-\mathcal{D}_{#1}^{#2}}
\newcommand\uldsv[2]{{}^\pm\!\mathcal{D}_{#1}^{#2}}
\newcommand\udsv[2]{{}^+\mathcal{D}_{#1}^{#2}}

\newcommand\proj[1]{{\mathbb{P}_{#1}^1}}
\newcommand\tG{{\tilde G}}
\newcommand\tGm{{\tilde\Gm}}
\newcommand\tZ{{\tilde Z}}

\newcommand\V{{\mathcal V}}
\newcommand\pD{{}^+{\mathcal D}}
\newcommand\mD{{}^-{\mathcal D}}
\newcommand\hypg[2]{\,{}_{#1}F_{\!#2}}
\newcommand\T{{\mathbf{T}}}
\newcommand\Gf{\Gamma} 
\newcommand\pol{{\mathrm{pol}}}
\newcommand\prs{{\mathrm{ps}}}

\newcommand\fd{{\mathfrak F}}

\newcommand\grf[2]{
\includegraphics[width=#1cm]{#2.pdf}}

\iflabels
  
  \usepackage{showlabels}
\fi

\iftxfonts
\usepackage[varg]{txfonts}
\fi

\ifsrcltx
\usepackage[active]{srcltx}
\fi

\def\be#1\ee{\begin{equation}#1\end{equation}}
\def\bad#1\ead{\be\begin{aligned}#1\end{aligned}\ee}
\def\badl#1#2\eadl{\be\label{#1}\begin{aligned}#2\end{aligned}\ee}
\renewcommand\={\,=\,}

\newtheorem{thm}{Theorem}[section]
\newtheorem{prop}[thm]{Proposition}
\newtheorem{cor}[thm]{Corollary}
\newtheorem{lem}[thm]{Lemma}

\theoremstyle{definition}
\newtheorem{defn}[thm]{Definition}

\renewcommand\setminus{\smallsetminus}

\usepackage{color}

\definecolor{blue}{rgb}{0,0,1}
\definecolor{red}{rgb}{1,0,0}
\definecolor{green}{rgb}{0,.6,.2}
\definecolor{purple}{rgb}{1,0,1}

\long\def\red#1\endred{{\color{red}#1}}
\long\def\blue#1\endblue{{\color{blue}#1}}
\long\def\purple#1\endpurple{{\color{purple}#1}}
\long\def\green#1\endgreen{{\color{green}#1}}

\makeatletter
\newcommand\matc[4]{\left( {#1\@@atop #3}{#2\@@atop #4}\right)}
\newcommand\matr[4]{\left( {\hfill #1\@@atop\hfill #3}{\hfill
#2\@@atop\hfill #4}\right)}
\makeatother

\newcommand\rmrk[1]{\medskip\par\noindent\emph{#1. }}

\numberwithin{equation}{section}

\begin{document}
\title{Modular cocycles and cup product}
\author{Roelof Bruggeman}
\address{Mathematisch Instituut Universiteit Utrecht, Postbus 80010,
3508 TA \ Utrecht, Nederland}
\email{r.w.bruggeman@uu.nl}

\author{YoungJu Choie}
\address{Dept.~of Mathematics and PMI, POSTECH, Pohang, Korea 1790--784}
\email{yjc@postech.ac.kr}

\begin{abstract}We extend to positive real weights Haberland's formula
giving a cohomological description of the Petersson scalar product of
modular cusp forms of positive even weight. This relation is based on
the cup product of an Eichler cocycle and a Knopp cocycle.\\
We also consider the cup product of two Eichler cocycles attached to
modular forms. In the classical context of integral weights at least
$2$ this cup product is uninteresting. We show evidence that for real
weights this cup product may very well be non-trivial. We approach the
question whether the cup product is a non-trivial coinvariant by
duality with a space of entire modular forms. Under suitable conditions
on the weights this leads to an explicit triple integral involving
three modular forms. We use this representation to study the cup
product numerically.

\end{abstract}



\maketitle

\section{Introduction}

Due to Eichler \cite{Ei57} and Knopp \cite{Kn74} we have linear
injective maps from spaces of cusp forms of a given weight $k$ into a
parabolic cohomology group $H^1_p(\Gm;V)$, where $V$ is a
finite-dimensional $\Gm$-module if $k\in \ZZ_{\geq 2}$ (Eichler)
and $V$ is a specific infinite dimensional module if the weight $k$ is
real
(Knopp).

The second cohomology group $H^2(\Gm;V)$ is zero for all modules $V$ and
for all cofinite discrete subgroups of $\SL_2(\RR)$. However for
 parabolic cohomology we have
\be\label{H2pci} H^2_p(\Gm;V) \cong V_\Gm = V \bigm/ \sum_{\gm\in \Gm }
V|(1-\gm)\,, \ee
the space of coinvariants, provided $\Gm$ has cusps; see
\cite[(11.9)]{BLZ}.
(We work with right $\Gm$-modules that are vector spaces over $\CC$, and
denote the action by a slash.)

The cohomology group $H^2_p(\Gm;\CC)\cong \CC$ is used in the
cohomological description of the Petersson scalar product on spaces of
cusp forms of integral weights at least~$2$ by Haberland \cite{Ha}. The
cup product is used to go from a pair of $1$-cocycles associated to
cusp forms to $H^2_p(\Gm;\CC)\cong \CC$. This method has been extended
and used by Kohnen and Zagier \cite[p.~244--245]{KZ}, Zagier
\cite[Theorem~1]{Za85}, Cohen \cite{Coh},  Pasol and Popa \cite{PP}  
and Choie, Park and Zagier
\cite{CPZ}. The Petersson scalar product of Maass cusp forms can also
be formulated in terms of a cup product \cite[Theorem 19.1]{BLZ}.

Here we investigate the cup product of two cohomology classes associated
to two holomorphic modular forms that may have different real weights
(one of them should be a cusp form). In general this is an rather
abstract exercition providing a coinvariant in a complicated
$\Gm$-module, which is the tensor product $M_1\otimes M_2$ of two
modules of infinite dimension. With an intertwining operator
$M_1\otimes M_2 \rightarrow N$ we obtain from the cup product a
coinvariant in a simpler module $N$ that may be easier to study. We
proceed with two choices. One choice leads to the extension to real
weights of the cohomological description of the Petersson scalar
product of cusp forms. The other choice leads to a trilinear form on
the product of three spaces of modular forms.

We restrict our attention the the full modular group $\Gm=\SL_2(\ZZ)$.
We work with holomorphic modular forms of real weight, with a
corresponding multiplier system. To a modular form $f$ of weight
$r\in \RR$ and multiplier system $v$ we may associate to types of
$1$-cocycles. The first type is the homogeneous Eichler cocycle
\be c_f(z_1,z_2;t) \= \int_{\tau=z_1}^{z_2} f(\tau)\, (\tau-t)^{r-2}\,
d\tau \ee
where the variable $t$ runs through the lower half-plane $\lhp$, such
that $(z-t)^{r-2}$ is well defined. The points $z_1$ and $z_2$ are in
the upper half-plane $\uhp$; if $f$ is a cusp form we may take $z_1$
and/or $z_2$ equal to a cusp of~$\Gm$.

The other type is the Knopp cocycle, which depends on the modular form
in a conjugate-linear way.
\be c^K_f(z_1,z_2;z) \= \int_{\tau=z_1}^{z_2}\overline{ f(\tau)} \,
(\bar\tau-z)^{r-2}\, d\tau \,.\ee
This cocycle has values in the holomorphic functions on the upper
half-plane. It is defined for $z_1,z_2\in \uhp$; if $f$ is a cusp form
the points $z_1$ and $z_2$ may be cusps.

When dealing with cocycles it is important to indicate the module in
which they take their values. Where possible we follow the notations
and conventions of \cite{BCD}, which we recall in more detail in
Section~\ref{sect-afcoc}. For the Eichler cocycles we use two modules
$\dsv{v,2-r}\om \subset \dsv{v,2-r}\infty$ of holomorphic functions on
the lower half-plane $\lhp$ in which the group $\Gm$ acts with weight
$2-r$ and multiplier system $v$. The elements of $\dsv{v,2-r}\om$
extend holomorphically into the upper half-plane, the elements of
$\dsv{v,2-r}\infty$ extend to functions in $C^\infty( \lhp\cup\RR)$.
Both extensions also satisfy some condition at~$\infty$. If
$z_1,z_2\in \uhp$ then the values of $c_f$ are in $\dsv{v,2-r}\om$. If
 cusps are involved we need to use $\dsv{v,2-r}\infty$. The Knopp
 cocycles take values in similar modules
$\udsv{v^{-1},2-r}\om \subset \udsv{v^{-1},2-r}\infty$ of holomorphic
functions on $\uhp$. See Section~\ref{sect-afcoc} for a further
discussion.

The association $f\mapsto c_f$ induces an injective linear map from the
space $\cf r(\Gm,v)$ of holomorphic modular cusp forms of weight $r>0$
and corresponding multiplier system $v$ to the parabolic cohomology
group $H^1_p(\Gm;\dsv{v,2-r}\infty)$, and $f\mapsto c^K_f$ induces a
conjugate linear map
$\cf r(\Gm;v) \rightarrow H^1_p(\Gm;\udsv{v^{-1},2-r}\infty)$.

The space of cusp forms $\cf r(\Gm,v)$ is contained in the much larger
space $A_r(\Gm,v)$ of all modular forms of weight $r$ and multiplier
system~$v$
(without any condition on the growth at the cusps). This space has
infinite dimension for all real weights. The association
$f\mapsto c_f$, respectively $f\mapsto c^K_f$, induces a linear map
$A_r(\Gm,v) \rightarrow H^1(\Gm;\dsv{v,2-r}\om)$, respectively a
conjugate linear map
$A_r(\Gm,v) \rightarrow H^1(\Gm;\udsv{v^{-1},2-r}\om)$. If
$r\in \ZZ_{\geq 2}$ the kernels of these maps have infinite dimension.
For $r\in \RR\setminus \ZZ_{\geq 2}$ these maps are
injective.\smallskip

Let $M_1$ and $M_2$ be $\Gm$-modules. As we will discuss in
Section~\ref{sect-cp} in more detail, there is a cup product
construction in parabolic cohomology that extends, under some
conditions, to a bilinear map
\be \cup: H^1(\Gm;M_1) \times H^1_p(\Gm;M_2) \longrightarrow
H^2_p(\Gm;M_1\otimes M_2)\,.\ee
In parabolic cohomology there is a (not canonical) isomorphism
$H^2_p(\Gm;V) \cong V_\Gm$. The space of coinvariants $V_\Gm$ has been
defined in~\eqref{H2pci}.

We apply this with $M_1 = \dsv{v_1,2-r_1}\om$ and $M_2$ either the
$\Gm$-module $\dsv{v_2,2-r_2}\infty$ or the $\Gm$-module
$\udsv{v_2^{-1},2-r_2}\infty$. Composition with the maps from modular
forms to Eichler or Knopp cocycles we get maps
\badl{CEE} C_{EE}&: A_{r_1}(\Gm,v_1) \times \cf{r_2}(\Gm,v_2)
\rightarrow \Bigl( \dsv{v_1,2-r_1}\om \otimes
\dsv{v_2,2-r_2}\infty\Bigr)_\Gm\,,\\
C_{EE}&(f_1,f_2) \= [c_{f_1}]\cup[c_{f_2}]\,, \eadl
which is bilinear, and
\badl{CEK} C_{EK}&: A_{r_1}(\Gm,v_1) \times \cf{r_2}(\Gm,v_2)
\rightarrow \Bigl( \dsv{v_1,2-r_1}\om \otimes
\udsv{v_2^{-1},2-r_2}\infty\Bigr)_\Gm\,,\\
C_{EK}&(f_1,f_2) \= [c_{f_1}]\cup[c^K_{f_2}]\,, \eadl
which is linear in $f_1$ and conjugate linear in~$f_2$.

The tensor products in which these maps take their values are large and
have a complicated structure.  Simplifications are possible.

First we simplify \eqref{CEE}. Since elements of $\dsv{v_1,2-r_1}\om $
and of$\dsv{v_2,2-r_2}\infty$ are holomorphic functions on $\lhp$ their
product is also holomorphic on~$\lhp$. Considering actionsand the behavior at the
boundary  we check that multiplication of functions
induces a linear intertwining operator of $\Gm$-modules
\be M: \dsv{v_1,2-r_1}\om \otimes \dsv{v_2,2-r_2}\infty \rightarrow
\dsv{v_1v_2,4-r_1-r_2}\infty\,.\ee
We restrict our attention to the composite bilinear map
\be M\circ C_{EE} :A_{r_1}(\Gm,v_1) \times \cf{r_2}(\Gm,v_2) \rightarrow
\bigl( \dsv{v_1v_2,4-r_1-r_2}\infty \bigr)_\Gm\,,\ee
The module $ \dsv{v_1v_2,4-r_1-r_2}\infty$ still has infinite dimension,
but we are able to do some explicit work.

In the case of an Eichler and a Knopp cocycle  in \eqref{CEK},
there is not such a product construction. The sole case that
we can handle further is the case that the weights and multiplier
systems are equal. There is a $\Gm$-invariant bilinear form
$[\cdot,\cdot]_{2-r}$ on
$\dsv{v,2-r}\infty \times \udsv{v^{-1},2-r}\infty$,  which induces a
$\Gm$-equivariant linear map $D:
\dsv{v,2-r}\infty \times \udsv{v^{-1},2-r}\infty\rightarrow \CC$. The
image is the trivial $\Gm$-module $\CC$, and hence $\CC_\Gm=\CC$. With
use of $D \circ C_{EK}$ we can extend Haberland's cohomological
description of the Petersson scalar product to all real weights. See
Section~\ref{sect-Hab}.\medskip

In case of two Eichler cocycles the coinvariant $M C_{EE}(f_1,f_2)$ can
be represented by many elements of $\dsv{v_1,v_2,4-r_1-r_2}\infty$. In
Section~\ref{sect-cp} we arrive at the following representative that we
consider convenient:
\be \cp(f_1,f_2) \= c_{f_1}(\rho-1,\rho) \cdot c_{f_2}(i,\infty)\,,\ee
with $\rho=e^{\pi i /3}$ and $\rho-1=e^{2\pi i/3}$ the vertices in
$\uhp$ of the standard fundamental domain of the modular group.

The question is whether $\cp(f_1,f_2)$ represents the trivial
coinvariant. In other words, whether there are $a$ and $b$ in the
module $M=\dsv{v_1,v_2,4-r_1-r_2}\infty$ such that
$\cp(f_1,f_2) = a|(1-S) + b|(1-T)$ with the generators $T=\matc 1101$
and $S=\matr0{-1}10$ of $\Gm$.

In the classical situation of weights $r_1,r_2\in \ZZ_{\geq 2}$ it turns
out that $\cp(f_1,f_2)$ represents the trivial coinvariant in
$\dsv{v_1v_2,4-r_1-r_2}\infty$. See Subsection~\ref{sect-class}. This
holds for the modular group; for subgroups of finite index the
situation is a bit more complicated. We will give numerical evidence
that there are $f_1$ and $f_2$ (with real non-integral weights) for
which $\cp(f_1,f_2)$ represents a non-trivial coinvariant. (This will
be done in Subsection~\ref{sect-num}.)

In Proposition~\ref{prop-triv} we will show that there are modules
$M_1\supset \dsv{v_1,v_2,4-r_1-r_2}\infty$ in which the coinvariant
represented by $\cp(f_1,f_2)$ is trivial. This is not surprising, since
every cocycle is a coboundary if one works with a suitably large
module.
\medskip

The further investigation of the coinvariant represented by
$\cp(f_1,f_2)$ we use the $\Gm$-invariant bilinear form
$[\cdot,\cdot]_r$ mentioned  above. In
Theorem~\ref{thm-dual} we will show that there is for real $r$ and
corresponding multiplier system a bilinear form $[\cdot,\cdot]_{2-r}$
on $\udsv{v^{-1},2-r}{-\infty} \times \dsv{v,2-r}\infty$. The
$\Gm$-module
$\udsv{v^{-1},2-r}{-\infty} \supset \udsv{v^{-1},2-r}\infty$ consists
of the holomorphic functions on~$\uhp$ with at most polynomial growth
at the boundary.

Suppose that we have an $\Gm$-invariant
$\bt\in \bigl(\udsv{v_1^{-1}v_2^{-1},4-r_1-r_2}{-\infty} \bigr)^\Gm$.
Then
\[ \bigl[ \bt, a|(1-T)+b|(1-S) \bigr]_{2-r} \=0\]
for all $a,b \in \dsv{v,2-r}\infty$. So if
$\bigl[ \bt, \cp(f_1,f_2) \bigr]_{2-r}\neq 0$ then we know that
$\cp(f_1,f_2)$ represents a non-trivial coinvariant.

The nice fact is that we know
$\bigl(\udsv{v_1^{-1}v_2^{-1},4-r_1-r_2}{-\infty} \bigr)^\Gm$. It is
the space of entire modular forms
$M_{4-r_1-r_2}(\Gm,v_1^{-1}v_2^{-1})$. We put $v_3=v_1^{-1}v_2^{-1}$
and $r_3=4-r_1-r_2$, and define a trilinear form
\be \T(f_1,f_2,f_3) \= \bigl[ f_3,\cp(f_1,f_2) \bigr]_{r_3}\ee
on $A_{r_1}(\Gm,v_1) \times \cf{r_2}(v_2) \times M_{r_3}(\Gm,v_3)$. If
for given $f_1,f_2$ we can find $f_3$ such that
$\T(f_1,f_2,f_3) \neq 0$, then we know that $\cp(f_1,f_2)$ represents a
non-trivial coinvariant.

The drawback is that we do not know whether all $\Gm$-invariant linear
forms on $\dsv{v_3^{-1},r_3}\infty$ arise from $M_{r_3}(\Gm,v_3)$. So
it may happen that $\cp(f_1,f_2)$ represents a non-trivial coinvariant
and still $\T(f_1,f_2,f_3)$ vanishes for all entire modular
forms~$f_3$.

To investigate the trilinear form $\T$ we can unravel all definitions,
and arrive at a complicated description. See Proposition~\ref{prop-4i}.
We obtained the following, slightly simpler, description.
\begin{thm}\label{thm-tri}Let
\be\label{f123} f_1 \in A_{r_1}(\Gm,v_1)\,,\quad f_2\in
\cf{r_2}(\Gm,v_2)\,,\quad f_3\in M_{r_3}(\Gm,v_3)\ee
with
\be\label{rvcond0}
r_1+r_2+r_3\=4\,,\quad r_1<2, \quad 0<r_2<2\,,
\quad v_1v_2v_3 \= 1\,. \ee
Then
\begin{align}\nonumber
\T(f_1,f_2,f_3)&\= \frac{(-2i)^{r_3}\,\Gf(r_3)}{\Gf(2-r_1)\,
\Gf(2-r_2)}\, \int_{\tau_1=\rho-1}^{\rho} f_1(\tau_1)\,
\int_{\tau_2=i}^{\infty} \, f_2(\tau_2) \\
 \label{tiu} &\qquad\hbox{} \cdot
 \int_{u=0}^1 f_3\bigl( \tau_1+u(\tau_2-\tau_1)\bigr)\,
 u^{1-r_2}\,(1-u)^{1-r_1}\, du\, d\tau_2\, d\tau_1
 \\\nonumber
 &\=\frac{(-2i)^{r_3}\,\Gf(r_3)}{\Gf(2-r_1)\, \Gf(2-r_2)}\,
\int_{\tau_1=\rho-1}^{\rho} f_1(\tau_1)\, \int_{\tau_2=i}^{\infty} \,
 f_2(\tau_2) \\
 \label{tii}
 &\qquad\hbox{} \cdot
 \int_{z=\tau_1}^{\tau_2}\, f_3(z) \,\Bigl(
 \frac{z-\tau_1}{\tau_2-\tau_1}\Bigr)^{1-r_2}\, \Bigl(
 \frac{\tau_2-z}{\tau_2-\tau_1}\Bigr)^{1-r_1}\, \frac{
 dz}{\tau_2-\tau_1}\, d\tau_2\,d\tau_1\,.
\end{align}
 The paths of integration for $\tau_1$ and $\tau_2$ are geodesic 
segments.  
In \eqref{tiu} the path of integration   over $u$  
 is the real interval $[0,1]$; in
\eqref{tii} the path of $z$ is any path from $\tau_1$ to $\tau_2$ that
does not cross $\ell_{\tau_1,\tau_2}\setminus s_{\tau_1,\tau_2}$, where
$s_{\tau_1,\tau_2}$ is the geodesic segment from $\tau_1$ to $\tau_2$,
and $\ell_{\tau_1,\tau_2}$ is the geodesic line through $\tau_1$
and~$\tau_2$.
\end{thm}
In \eqref{tii} we may just take the geodesic segment from $\tau_1$ to
$\tau_2$ as the path of integration for $z$.

In \S\ref{sect-num} we formulate in Proposition~\ref{prop-triFc} the
trilinear form $\T$ in terms of Fourier coefficients of the~$f_j$. This
formulation is suitable for numerical computations.

\rmrk{Overview of the paper}Section~\ref{sect-afcoc} recalls definitions
over various concepts in some more details that in this introduction.
Section~\ref{sect-cp} discusses the cup product.

In both cases that we consider (two Eichler cocycles, and the
combination of an Eichler cocycle and a Knopp cocycle) we use the
duality theorem Theorem~\ref{thm-dual} in Section~\ref{sect-dual}. To
prove the duality theorem we use principal series representations of
the universal covering group of $\SL_2(\RR)$. This requires further
definitions and discussion, which we have put at the end of the paper
in Section~\ref{sect-ucgps}.

In Section~\ref{sect-Hab} we extend to real positive weights Haberland's
relation between the Petersson scalar product and the cup product.

Sections \ref{sect-clp}--\ref{sect-tri} study the coinvariant
represented by $\cp(f_1,f_2)$. Section \ref{sect-clp} shows that in the
classical context this coinvariant is uninteresting.
Section~\ref{sect-EE} discusses the definition of the trilinear form
$\T$ and show explicitly that $\cp(f_1,f_2)$ represents the trivial
coinvariant over the module of all holomorphic functions on~$\lhp$.

The main work to prove Theorem~\ref{thm-tri} is done in
Section~\ref{sect-trip}. In Section~\ref{sect-num} we mention how we
approach the study of $\T(f_1,f_2,f_3)$ numerically.

\rmrk{Acknowledgements}The second author is partially supported by NRF
$2018R1A4A$ $1023590$ and NRF $2017R1A2B$ $2001807$.

\section{Modular forms and cohomology}\label{sect-afcoc}
\rmrk{Modular forms}By a holomorphic modular form $f$ of \emph{weight}
$r\in \RR$ with \emph{multiplier system} $v$ we mean a holomorphic
function $f:\uhp\rightarrow\CC$ such that
\be \label{mf}f(\gm z ) \= v(\gm)\, (cz+d)^r f(z) \qquad\text{for all }
\gm =\matc abcd \in \Gm \= \SL_2(\ZZ)\,.\ee
We take $\arg(cz+d)\in (-\pi,\pi]$. The multiplier system is a function
$v:\Gm \rightarrow \CC^\ast$ such that non-zero solutions of this
equation are possible. For the modular group we use $v=v[p]$ with
$p \equiv r \bmod 2$, where $v[p]$ is the multiplier system of the
$(2p)$-th power of the Dedekind eta-function:
\be\label{vpdef} v[p]\matc abcd \= \frac{\eta^{2p}(\gm z)}{(cz+d)^p
\eta^{2p}(z)}\,.\ee
The transformation behavior of the Dedekind eta function has a
multiplier system with values in the $24$-th roots of unity. See, eg,
\cite[Chap.~IX]{La}. Hence the multiplier system $v[p]$ depends only on
$p\bmod 12$. Any multiplier system suitable for weight $r$ satisfies
\be v\matr{-1}00{-1} \= e^{-\pi i r}\,.\ee

We denote by $A_r(v)=A_r(\Gm,v)$ the space of all holomorphic $f$
satisfying~\eqref{mf}. We do not impose any further restriction; so
$A_r(v)$ has infinite dimension. \emph{Entire modular forms} satisfy
the condition $f(z) = \oh(1)$ as $\im z\uparrow\infty$; we denote by
$M_r(v)=M_r(\Gm,v)$ the resulting subspace of $A_r(v)$. This space is
finite dimensional and zero if $r<  0$. A further restriction is the
condition of quick decay $f(z) = \oh\bigl( (\im z)^{-a}\bigr)$ as
$\im z\uparrow\infty$. It defines the subspace
$\cf r(v)=\cf r(\Gm,v)\subset M_r(v)$ of \emph{cusp forms}, which is
zero if $r\leq 0$.

\rmrk{Actions} For each weight $r$ and corresponding multiplier system
$v$ there are right actions of $\PSL_2(\ZZ) = \Gm/\{1,-1\}$; the action
$|_{v,r}$ is on functions on the upper half-plane $\uhp$, and
$f|^-_{v^{-1},r}$ on functions on the lower half-plane $\lhp$ given for
$\gm =\matc abcd$ by
\badl{actions} f|_{v,r}\gm(z) &\= v(\gm)^{-1}\, (cz+d)^{-r}\, f(\gm
z)&\quad&(z\in \uhp)\,,\\
f|_{v^{-1},r}^- \gm (t) &\= v(\gm) \, (ct+d)^{-r} \, f(\gm t)&&
(t\in \lhp)\,. \eadl
It turns out to be convenient to use the argument conventions
$\arg(cz+d) \in (-\pi,\pi]$ for $z\in \uhp$, and
$\arg(ct+d)\in [-\pi,\pi)$ for $t\in \lhp$. Hence we use $v(\gm)^{-1}$
in the action on functions on~$\uhp$ and $v(\gm)$ in the action on
functions of $\lhp$. These representations of $\Gm$ are trivial on the
center $\{1,-1\}$ of $\Gm$. So they are in fact representations of
$\PSL_2(\ZZ) \cong \Gm/\{1,-1\}$.

The property \eqref{mf} of modular forms is invariance under the action
$|_{v,r}$ of~$\Gm$. See \cite[\S2.1]{BCD} for more information.

\rmrk{Eichler cocycles}For $f\in A_r(v)$, $t\in \lhp$, $z_1,z_2\in \uhp$
we put
\be\label{Eidef} c_f(z_1,z_2;t) \= \int_{z=z_1}^{z_2} f(z)\,
(z-t)^{r-2}\, dz\,,\ee
with $\arg(z-t) \in \bigl( -\pi/2,3\pi/2\bigr)$. This defines
holomorphic functions on $\lhp$ satisfying
\bad c_f(z_1,z_3) &\= c_f(z_1,z_2)+c_f(z_2,z_3)\,,\\
c_f(\gm^{-1}z_1,\gm^{-1}z_2) &\= c_f(z_1,z_2)|^-_{v,2-r}\gm\qquad(\gm\in
\Gm)\,. \ead
So this defines a homogeneous $1$-cocycle, called an \emph{Eichler
cocycle}. It has values in the holomorphic functions on $\lhp$ with the
action $|^-_{v,2-r}$. The corresponding inhomogeneous $1$-cocycle is
$\gm \mapsto c_f(\gm^{-1}z_0,z_0)$; it depends on the choice of a base
point $z_0\in \uhp$. See for instance \cite[\S6.1]{BLZ} for a
discussion of cohomology based on homogeneous cocycles.

The values of Eichler cocycles are holomorphic functions on $\lhp$
satisfying further properties. They are in the $\Gm$-module
$\dsv{v,2-r}\om$ of holomorphic functions $h$ on $\lhp$ that have a
holomorphic continuation to a neighborhood of $\RR$ in $\CC$ and for
which $t\mapsto (i-t)^{2-r}\, h(t)$ is holomorphic in $\frac{-1}t$ on a
neighborhood of $\frac {-1}t$ in~$\CC$. The association $f\mapsto c_f$
induces a linear map
 \be \label{AH1} A_r(v) \longrightarrow H^1(\Gm;\dsv{v,2-r}\om)\ee
which is injective if $r\in \RR \setminus \ZZ_{\geq 2}$ \cite[Theorem
A]{BCD}.

For a cusp form $f\in \cf r(v)$ we can form
$c_f(z_1,z_2;t)= \int_{z=z_1}^{z_2} \, f(z)\, (z-\nobreak t)^{r-2}\,  dz,$
where $z_1$ and $z_2$ may be cusps. These cocycles take values in the
larger module $\dsv{v,2-r}\infty$ of holomorphic functions $h$ on
$\lhp$ for which $h$ has an extension to $C^\infty(\lhp\cup\RR)$ and
for which $t\mapsto  (i-t)^{2-r}\,h(t)$ is $C^\infty$ in $\frac{-1}t$
on a neighborhood of $0$ in $\lhp \cup\RR$. The association
$f \mapsto c_f$ induces an injective linear map
\be \label{SH1p}\cf r(v) \longrightarrow
H^1(\Gm;\dsv{v,2-r}\infty)\,.\ee
This map is bijective if $r\in \RR \setminus \ZZ_{\geq2}$, by Theorem~B
in \cite{BCD}. For $r\in \ZZ_{\geq 2}$ the cocycles take values in the
$(r-2)$-dimensional submodule $\dsv{v,2-r}\pol \subset \dsv{v,2-r}\om$
of polynomial functions on $\CC$ of degree at most $r-2$.

\rmrk{Knopp cocycles} The conjugate linear map $\iota$ given by
$\iota f (z) = \overline{f(\bar z)}$ interchanges functions on $\uhp$
and $\lhp$ and intertwines the actions $|_{v,r}$ and $|^-_{v^{-1}, r}$
(for real $r$ and corresponding multiplier system). This map can be used
to define $\udsv{v^{-1},2-r}\om\subset \udsv{v^{-1},2-r}\infty$ in
terms of the modules $\dsv{v,2-r}\om \subset \dsv{v,2-r}\infty$ with
analogous descriptions. If $r\in \ZZ_{\geq 2}$ we have also
$\udsv{v^{-1},2-r}\pol\subset \udsv{v^{-1},2-r}\om$ consisting of the
polynomial functions of degree at most $r-\nobreak2$.

The Knopp cocycle
\be c^K_f(z_1,z_2) \= \iota c_f(z_1,z_2) \ee
induces conjugate linear maps
\bad &A_r(v) \rightarrow H^1(\Gm;\udsv{v^{-1},2-r}\om)\,,\\
&\cf r(v) \rightarrow H^1_p(\Gm;\udsv{v^{-1},2-r}\infty)\,. \ead
Actually, Knopp \cite{Kn74} considers the resulting map with values in
the larger module $\udsv{v^{-1},2-r}{-\infty}$ of holomorphic functions
on $\uhp$ with at most polynomial growth:
\be h(z) \= \oh\bigl( (\im z)^{-B} \bigr) + \oh \bigl(|z|^{\s}\bigr)
\qquad z\in \uhp\, \text{ for some }B>0\text{ and }\s>0\,. \ee
Knopp and Mawi \cite{KM10} showed that
\be \cf r(v) \rightarrow H^1_p(\Gm;\udsv{v^{-1},2-r}{-\infty})\ee
is bijective for all $r\in \RR$.

\section{Cup product of modular cocycles}\label{sect-cp}
We recall the cup product construction in parabolic cohomology. Applied
to two $1$-cocycles with values in possibly different modules it gives
a coinvariant in the tensor product of the two modules.

Applied to pairs of cocycles attached to modular forms this construction
yields coinvariants in a module that may be large and complicated to
investigate. We discuss various choices in which one obtains
coinvariants in simpler modules.

\subsection{Cup product}We base our discussion of the parabolic cup
product on the discussion in \cite[\S19]{BLZ}, where more details can
be found.

\begin{figure}[htp]
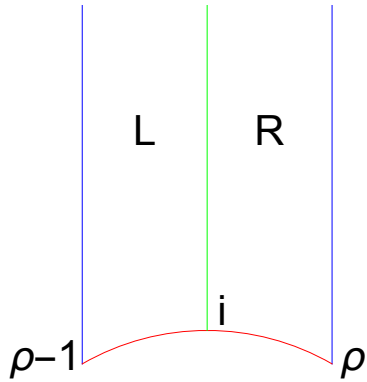

\begin{center}\grf5{cpmc-tessf}\end{center}
\caption{The standard fundamental domain of the modular group $\Gm$ as
the union of two triangles. The tesselation of the upper half-plane
formed by the $\Gm$-translates of these triangle is the basis of the
cup product construction that we use.} \label{fig-tessf}
\end{figure}
We use the tesselation of the upper half plane given by the
$\Gm$-translates of the triangles $L$ and $R$ in
Figure~\ref{fig-tessf}. This tesselation determines a resolution with
which we can compute parabolic cohomology, as discussed in
\cite[\S12.1]{BLZ}. It is based on the spaces $\CC[X_i]$, where $X_0$
is the set of vertices of the tesselation, $X_1$ a set of oriented
edges, and $X_2$ the set of faces.

In \cite[\S19.2]{BLZ} the construction of the cup product is based on a
diagonal approximation. In the context of general cofinite groups a
diagonal approximation as in Table 19.1 of \cite{BLZ} is used. Here we
can proceed   in a simpler way,  
and determine for $i=0,1,2$ the diagonal
approximation
\be \dt_i : \CC[X_i] \rightarrow \sum_{j=0}^i \CC[X_j] \otimes
\CC[X_{i-j}]\ee
as a $\CC[\Gm]$-linear map. It has to satisfy the compatibility
relations $\partial_j \dt_j = \dt_{j-1}\partial_j$ with the boundary
map~$\partial_j$.
In dimension $0$ we determine it by
\be \dt_0 (P) = (P) \otimes (P) \qquad (P\in X_0)\,.\ee
In dimension $1$ we take on the basis edges $e_{\rho,\infty}$,
$e_{\rho-1,i}$ and $e_{i,\infty}$
\be \dt_1 e_{P,Q} \= (P) \otimes e_{P,Q} + e_{P,Q}\otimes Q\,,\ee
use $\dt_1 e_{Q,P}=-\dt_1 e_{P,Q}$, and extend $\dt_1$
$\CC[\Gm]$-linearly. Guided by the diagonal embedding in Table 19.1
in~\cite{BLZ} we take

\badl{dt2def} \dt_2 R &\= (\rho) \otimes R + e_{\rho,i}\otimes
e_{i,\infty }+ R \otimes (\infty)\,,\\
\dt_2 L &\= (\rho-1)\otimes L - e_{\rho-1,i} \otimes e_{i,\infty} + L
\otimes(\infty)\,, \eadl
and extend this $\CC[\Gm]$-linearly. It takes some work to check the
compatibility relation $\partial_2\dt_2=\dt_1\partial_2$ with the
boundary maps.\smallskip

The cup product of cochains $c_1 \in C^p(\CC[X_\cdot];V)$
and $c_2\in C^q(\CC[X_\cdot;W)$ is defined by
\be (c_1\cup c_2)(x) \= - (c_1\otimes c_2)\,\bigl( \dt_{p+q}(x)\bigr)
\qquad\text{for }x\in \CC[X_{p+q}]\,.\ee
The tensor $c_1\otimes c_2$ sees only the component
$\CC[X_p]\otimes \CC[X_q]$ of
\[(\CC[X_\cdot] \otimes \CC[X_\cdot])_{p+q} \= \bigoplus_{j}
\CC[X_j]\otimes \CC[X_{p+q-j}]\,.\]
This defines $c_1\cup c_2\in C^{p+q}(\CC[X_\cdot];V\otimes\nobreak W)$.
If both $c_1$ and $c_2$ are cocycles, then $c_1\cup c_2$ is a cocycle.
If one is a coboundary and the other a cocycle then the cup product is
a coboundary.

In particular for $p=q=1$ we are interested in
\bad
(c_1\cup c_2)&(L+R) \= -(c_1\otimes c_2)\left( e_{\rho,i}\otimes
e_{i,\infty}
- e_{\rho-1,i}\otimes e_{i,\infty}\right)\\
&\=-c_1(e_{\rho,i})\otimes c_2(e_{i,\infty})+c_1(e_{\rho-1,i}) \otimes
c_2(e_{i,\infty})
\\
&\=c_1(e_{\rho-1,\rho}) \otimes c_2(e_{i,\infty})\,. \ead

In principle this works when $c_1$ and $c_2$ are both parabolic
cocycles. However, $c_1$ is evaluated only on the interior edge
$e_{\rho-1,\rho}$, which makes sense for a cocycle that is not
parabolic. (See \S11.1 in \cite{BLZ}.)
Now \cite[equation (19.7)]{BLZ} states that this results in a
well-defined linear map
\be\label{H12cp} \cup : H^1(\Gm;V) \otimes H^1_p(\Gm;W) \rightarrow
H^2_p(\Gm; V \otimes W)\,.\ee
This is not true in the generality in which \cite{BLZ} states it. We
need the condition that the space of invariants $W^\pi$ is zero for all
parabolic $\pi\in \Gm$. For the modular group it suffices that
$W^T=\{0\}$.

\subsection{Application to modular forms}\label{sect-amf}We apply this
to the cocycles attached to two modular forms, $f_1\in A_{r_1}(v_1)$,
$f_2\in \cf{r_2}(v_2)$, with real weights $r_j$ and corresponding
multiplier systems, with $r_2>0$. We take the cocycle $c_1 = c_{f_1}$
representing a cohomology class in $H^1(\Gm;\dsv{v_1,2-r_1}\om)$, and
either the parabolic cocycle $c_2=c_{f_2}$ representing a class in
$H_p^1(\Gm;\dsv{v_2,2-r_2}\infty)$ or $c_2 = c^K(f_2)$ representing a
class in $H^1_p(\Gm;\udsv{v_2^{-1},2-r_2}\infty)$. We need to impose
the condition $r_2\not\in \ZZ_{\geq 2}$ to have $W^T=\{0\}$; see
\cite[Lemma 12.1]{BCD}. If $r_2\in \ZZ_{\geq 2}$ we assume that
$f_1\in \cf{r_1}(v_1)$ as well, to have a well-defined cup product.

In this way the cup product gives coinvariants in the $\Gm$-modules
\[ \dsv{v_1,2-r_1}\om\otimes \dsv{v_2,2-r_2}\infty \text{ and }
\dsv{v_1,2-r_1}\om\otimes \udsv{v_2^{-1},2-r_2}\infty\,,\]
or if $r_2\in \ZZ_{\geq 2}$ in
\[ \dsv{v_1,2-r_1}\om\otimes \dsv{v_2,2-r_2}\pol \text{ and }
\dsv{v_1,2-r_1}\om\otimes \udsv{v_2^{-1},2-r_2}\pol\,.\]
These tensor products are large complicated modules. It is hard to
determine what it means to have a representative of a coinvariant in
such a module. There is hope to get more information if there is an
intertwining operator from the tensor product to a simpler module.

We restrict our attention to a number of situations, in which an
intertwining operator to a simpler module can be found.
\begin{enumerate}
\item[EE.]
$M:\dsv{v_1,2-r_1}\om\otimes \dsv{v_2,2-r_2}\infty \rightarrow
\dsv{v_1v_2,4-r_1-r_2}\infty$ by multiplication of functions, under the
assumption that $r_2\not\in \ZZ_{\geq 2}$. The image of the cup product
$c_{f_1}\cup c_{f_2}$ is a coinvariant in
$\dsv{v_1v_2,4-r_1-r_2}\infty$ represented by the product
\be\label{cpd} \cp (f_1,f_2) \= c_{f_1}(\rho-1,\rho) \cdot
c_{f_2}(i,\infty)\ee
for $f_1\in A_{r_1}(v_1)$ and $f_2\in\cf{r_2}(v_2)$.

\item[cEE.]
$M:\dsv{v_1,2-r_1}\pol\otimes \dsv{v_2,2-r_2}\pol \rightarrow
\dsv{v_1v_2,4-r_1-r_2}\pol$ by multiplication of polynomial functions.
For $f_j \in \cf{r_j}(v_2)$ with $r_1,r_2\in \ZZ_{\geq 2}$ we get the
coinvariant in $\dsv{v_1v_2,4-r_1-r_2}\pol$ represented by the product
$\cp(f_1,f_2)$ in~\eqref{cpd}.

\item[cEK.]
$D\udsv{v^{-1},2-r} \infty \otimes \dsv{v,2-r}\infty  \rightarrow \CC$
by a duality to be discussed in Section~\ref{sect-dual}, for
$r=r_1=r_2>0$, $v=v_1=v_2$. To two cusp forms $f_1,f_2$ in the same
space $\cf r(v)$ is associated a number
$\bigl[ c_{f_2}^K(i, \infty),c_{f_1}(\rho-1,\rho) \bigr]_{2-r}$.
\end{enumerate}

We discuss case cEE in Sections \ref{sect-clp}, and show in
Section~\ref{sect-Hab} how case cEK leads to a generalization of the
relation between cup product and Petersson scalar product.

Case EE is the subject of Sections \ref{sect-EE} and~\ref{sect-tri},
where we use the duality theorem in Section~\ref{sect-dual} to go from
cup product to   a    fourfold and 
  a   triple integral.

\section{Duality theorem}\label{sect-dual}
In this section we formulate the duality theorem, and show that a space
of entire modular forms gives the continuous $\Gm$-invariant linear
forms on $\dsv{ v_1v_2,4-r_1-r_2}\infty$ that we will use for further
study of the cup product.

\subsection{$\Gm$-modules of holomorphic functions}\label{sect-mds}For a
real weight $r$ and a corresponding multiplier system we discussed in
Section~\ref{sect-afcoc} various space of holomorphic functions on
which $\Gm$ acts, related by the conjugate linear transformation
$\iota$ given by $(\iota f)(z) = \overline{f(\bar z)}$.
\bad \xymatrix{ \dsv{v^{-1},r}\om \ar@{<->}[d]^{\iota} \ar@{^{(}->}[r] &
 \dsv{v^{-1},r}\infty \ar@{<->}[d]^{\iota} \ar@{^{(}->}[r] &
\dsv{v^{-1},r}{-\infty} \ar@{<->}[d]^{\iota} \ar@{^{(}->}[r] &
 \dsv{v^{-1},r}{-\om} \ar@{<->}[d]^{\iota}
\\
\udsv{v,r}\om \ar@{^{(}->}[r] & \udsv{v,r}\infty \ar@{^{(}->}[r] &
\udsv{v,r}{-\infty} \ar@{^{(}->}[r] & \udsv{v,r}{-\om} } \ead
The modules in the top row consist of holomorphic functions on $\lhp$,
with the action $|^-_{v^{-1},r}$ of~$\Gm$, the modules in the bottom
row consist of holomorphic function on $\uhp$ with the action
$|_{v,r}$. If $r\in \ZZ_\leq 0$ there are the
$\left(|r|+1\right)$-dimensional submodules
$\dsv{v^{-1},r}\pol \subset \dsv{v^{-1},r}\om$ and
$\udsv{v,r}\pol \subset \udsv{v,r}\om$, also related by $\iota$. These
submodules consist of the polynomial functions of degree at most~$|r|$.

We put
\badl{fpm} f^-(t) &\= (i-t)^r\, f(t)&\quad& \text{for functions $f$ on
$\lhp$}\,,\\
 f^+(z) &\= (-i-z)^r\, f(z)&\quad& \text{for functions $f$ on $\uhp$}\,.
\eadl
With this notation we have the following characterizations:
\be \label{Dd}
\begin{array}{|cc|}\hline
\uldsv{v^{\pm 1},r} {-\om}& f^\pm \text{ is holomorphic on } \uhp^\pm\\
\uldsv{v^{\pm 1},r} {-\infty}&\exists_{B>0}\;f^\pm(z) = \oh\bigl( |\im
z|^{-B}\bigr)
+ \oh\bigl( |z|^B\bigr)\text{ on } \hp^\pm\\
\uldsv{v^{\pm 1},r} {\infty}&\text{extension }f^\pm \in C^\infty\bigl(
\hp^\pm \cup\proj\RR\bigr)\\
\uldsv{v^{\pm 1},r} {\om}&\text{holomorphic extension of $f^\pm$ to }U
\supset \hp^\pm \cup\proj\RR\\
\uldsv{v^{\pm 1},r}\pol&f^\pm\text{ polynomial function of } \frac{t\mp
i}{t\pm i}, \text{ degree at most $|r|$}
\\ \hline
\end{array}\ee

The actions on $\dsv{v,2-r}\infty$ and $\udsv{v^{-2},2-r}\infty$ are
continuous for the topology given by the supremum norms of all
derivatives on $\proj\RR$ of the extension of $f^\pm$ to
$\hp^\pm \cup \proj\RR$. The derivatives are taken with respect to
$\th = -\cot t$ for $t\in \proj\RR$ with $\th\in \RR\bigm/ \pi \ZZ$.

The representation spaces $\udsv{v,r}\om $ and $\dsv{v^{-1},r}\om$ are
the direct limits of spaces of bounded holomorphic functions $f^\pm$ on
neighborhoods $U_1$ of $\proj\RR$ in $\proj\CC$. The natural topology
is obtained by providing these spaces with the supremum norm on~$U_1$.

\subsection{Duality} \label{sect-dual} To formulate the duality theorem
we use linear operators~$\s_r$:
\bad \s_r&: \dsv{v^{-1},r}\om \rightarrow \dsv{v^{-1},r}\om&&\text{ if
}r\in \RR\setminus \ZZ_{\leq 0}\,,\\
\s_r&:\dsv{v^{-1},r}\pol \rightarrow \dsv{v^{-1},r}\om&&\text{ if }r\in
\ZZ_{\leq0}\,. \ead
For $r\in \RR\setminus \ZZ_{\leq0}$ we describe $\s_r$ in terms of the
functions $f^-: z\mapsto f(z) \,(i-\nobreak z)^r$ defined
in~\eqref{fpm}.
\be \label{sigrdef}\s_r f^-(z) \;:=\; \frac1 \pi \frac{z+i}{z-i}\,
\int_{\tau\in C} \, f^-(\tau)\,
\hypg21\Bigl(1,1;r;\frac{(\tau-i)(z+i)}{(\tau+i)(z-i)} \Bigr) \,
 \frac{d\tau}{\tau^2+1}\,,\ee
where the hypergeometric function is given on the unit disk by 
$\hypg21(a,b;c;z):=\sum_{n\geq 0} \frac{(a)_n(b)_n}{(c)_n}\frac{z^n}{n!}, $
and $(a)_n$ is the Pochhammer symbol. 
This function has an analytic extension to $\CC\setminus[1,\infty)$. 
The positively oriented curve $C$ in $\uhp$ encircles the domain of
$f^-$ and~$i$. The point $z$ is outside the curve~$C$. By definition of
$\dsv{v^{-1},r}\om$ the function is holomorphic on
$\proj\CC\setminus K$ for some compact set $K\subset \uhp$. Adapting
$C$ to $K$ we get a holomorphic function $\s_r f^-$ on
$\proj\CC \setminus K$.

If we would insist to formulate the relation $g^- = \s_r f^-$ in terms
of $f$ and $g$ in $\dsv{v^{-1},r}\om$ the formula would be more
complicated:
\[ g(z) \= (z-i)^{-r}\,\frac1 \pi \frac{z+i}{z-i}\, \int_{\tau\in C} \,
f(\tau)\, (i-\tau)^r\,
\hypg21\Bigl(1,1;r;\frac{(\tau-i)(z+i)}{(\tau+i)(z-i)} \Bigr) \,
 \frac{d\tau}{\tau^2+1}\,. \]
To see that this integral makes sense we would need to remember that
$\tau \mapsto f(\tau)\,\allowbreak (i-\nobreak\tau)^r$ extends holomorphically
from $\lhp$ across $\proj\RR$ into an region in~$\uhp$.

For $f_n^-(t) = (\frac{t-i}{t+i}\bigr)^n$, with $n\leq 0$, we will see
in Proposition~\ref{prop-lift} that
\be\label{srw} \s_r f_n^- \= \frac{|n|!}{(r)_{|n|}} \, f_{n-1}^-\,. \ee
If $r\in \ZZ_{\leq 0}$ then $(r)_{|n|}\neq 0$ for $r\leq n \leq 0$. We
use \eqref{srw} to define $\s_r$ for $r\in \ZZ_{\leq 0}$.

At this point the operators $\s_r$ seem arbitrary. In~\S\ref{sect-ps} we
will see   that they   arise naturally.

\begin{thm}\label{thm-dual} {\rm Duality theorem. }Let $r\in \RR $, with
corresponding multiplier system~$v$.
\begin{enumerate}
\item[i)] If $r\not\in \ZZ_{\leq 0}$ there is a non-degenerate
$\Gm$-invariant bilinear form $[\cdot,\cdot]_r$ on
$ \udsv{v,r}{-\om} \times \dsv{v^{-1},r}\om$ given by
\be\label{dualdef} \bigl[ h,f]_r \= \frac1\pi \int_{z\in C} h^+(z)
\,\s_r f^-(z) \, \frac{dz}{(z+i)^2}\,. \ee
The cycle $C$ in $\uhp$ is homotopic to $\proj\RR$ in the domain of
$\s_rf^-$, and encircles the point~$i$ once in the positive direction.

The value of the bilinear form does not change if we add to $\s_r f^-$
any holomorphic function on~$\uhp$.

\item[ii)]Let $r\not\in \ZZ_{\leq 0}$. If $h\in \udsv{v,r}{-\infty}$
then the linear form $f\mapsto [h,f]_r$ on $\dsv{v^{-1},r}\om$ extends
continuously to $\dsv{v^{-1},r}\infty$ for the natural topology on
$\dsv{v^{-1},r}\infty$ defined in \S\ref{sect-mds}.

\item[iii)] If $r\in \ZZ_{\leq 0}$ the relation \eqref{dualdef} defines
a $\Gm$-invariant bilinear form on
$\udsv{v,r}{-\om} \times \dsv{v^{-1},r}\pol$. Its restriction to
$\udsv{v,r}\pol \times \dsv{v^{-1},r}\pol$ is non-degenerate.

\item[iv)] In terms of the expansions
\[h^+(z) \= \sum_{n\geq 0} c_n \,\Bigl( \frac{z-i}{z+i}\Bigr)^n\,,\qquad
 f^-(z) \= \sum_{m\leq 0} d_m \,\Bigl( \frac{z-i}{z+i} \Bigr)^m\,, \]
with $d_m=0$ for $m<|r|$ if $r\in \ZZ_{\leq 0}$ the duality is given by
\be \bigl[ h,f\bigr]_r \= \sum_{n\geq 0} \frac{n!}{(r)_n}\, c_n\,
d_n\,.\ee
\end{enumerate}
\end{thm}

We will give the proof of the duality theorem in \S\ref{sect-dualps}, in
the context of principal series representations of the universal
covering group of $\SL_2(\RR)$.

\rmrk{Use of the duality theorem}At the end of Subsection~\ref{sect-amf}
we mentioned three cases in which we want to further investigate the
cup product of modular cocycles. In two of them we will use the duality
theorem.

In case cEK we associate to $f_1,f_2\in \cf r(v)$ the cup product
$c_{f_1}(\rho-1,\rho) \otimes c_{f_2}(i,\infty)$  in
$\dsv{v,2-r}\infty \otimes \udsv{v^{-1},2-r}\infty$.  We get a
$\Gm$-equivariant linear map from the tensor product to $\CC$ generated
by $v\otimes w\mapsto [w, v]_{2-r}$ for $r\not\in \ZZ_{\geq 2}$. If
$r\in \ZZ_{\geq 2}$ we use part iii) of the duality theorem. See
Section~\ref{sect-Hab}.

In the case EE we have
$c_{f_1}(\rho-1,\rho) \otimes c_{f_2}(i,\infty) \in
\dsv{v_1,2-r_1}\om \otimes \dsv{v_2,2-r_2}\infty$ representing the cup
product $c_{f_1}\cup c_{f_2}$. The linear map induced by multiplication
$v\otimes w\mapsto v\, w$ is an intertwining operator from the tensor
product to $\dsv{v_1v_2,4-r_1-r_2}\infty$. If we have $\Gm$-invariant
elements $u\in\udsv{v_1^{-1}v_2^{-1},4-r_1-r_2}{-\infty}$, we use the
linear form $\ph \mapsto \bigl[ u,\ph\bigr]_{2-r}$ resulting from
part~iii)
of the duality theorem to test whether the coinvariant represented by
$\cp(f_1,f_2)$ is non-trivial. See section~\ref{sect-EE}.

\section{Cup product and Petersson scalar product}\label{sect-Hab}

We consider now case EK in Subsection~\ref{sect-amf}. After
Theorem~\ref{thm-dual} we explained that to two cusp forms
$f_1,f_2\in \cf r(v)$ we associate the number
\[ \bigl[c_{f_2}^K(i,\infty),c_{f_1}(\rho-1,\rho)] _{2-r}\]
 in the trivial $\Gm$-module $\CC \cong \CC_\Gm$. 
 This number   turns out to be   a
 multiple of the Petersson scalar product. We extend Haberland's
relation in \cite{Ha} to real weights. 
\begin{thm}\label{thm-Hab}Let $r>0$ and let $v$ be a corresponding
multiplier system. For all $f_1,f_2\in \cf r(v)$
\[\bigl[c_{f_2}^K(i,\infty),c_{f_1}(\rho-1,\rho)] _{2-r} \= -2i \;\bigl(
f_1,f_2)_r \;:=\; -2i \int_{\Gm\backslash\uhp}
f_1(\tau)\,\overline{f_2(\tau)}\,y^r\, \frac{dx\, dy}{y^2} \,. \]
\end{thm}
\rmrk{Remarks}Haberland \cite[\S7]{Ha} gave this relation for even
positive weights. With the duality theorem Theorem~\ref{thm-dual}
available we can essentially follow Haberland's approach. As far as we
know the case of positive weights has not been considered in
generality. We note that Neururer's paper \cite{Neu}, Sections 2 and~3,
and Cohen's paper \cite{Coh18}, Section~3, come close to what we do
here.

\begin{proof}We start with the cup product and show that it is a
multiple of the Petersson scalar product.

\rmrk{Truncation} The function $c_{f_2}^K$ is in $\udsv{v^{-1},2-r}\om$,
since the integration is over a compact path in $\uhp$. However
$c_{f_1}\in \dsv{v,2-r}\infty$. We use as an approximation the element
$c_{f_1}(i,ia)\in \dsv{v,2-r}\om$.
\begin{lem}\label{lem-a}Let $f\in \cf r(v)$. Then
\[\lim_{a\uparrow\infty} c_{f}(i,ia) \= c_{f}(i,\infty)\]
in the natural topology on $\dsv{v,2-r}\infty$.
\end{lem}
\begin{proof}The topology on $\dsv{v,2-r}\infty$ is given by the norms
associating to $\ph$ on $\proj\RR$ the supremum on $\proj\RR$ of the
derivatives $\Bigl( \frac1{1+t^2}\partial_t \Bigr) ^N \ph^-(t)$.
We put $\ph(t) = c_{\!f}(i,\infty;t)$.
\begin{align*} \ph^-_a(t) \= i \int_{y=1}^a f(iy)
\,\Bigl(\frac{iy-t}{i-t} \Bigr)^{r_2-2}\, dy \= i \int_{y=1}^a f(iy)
\, \Bigl( \frac{iy/t-1}{i/t-1}\Bigr)^{r_2-2} \, dy\,,\\
\ph^-(t) \= i \int_{y=1}^\infty f(iy) \,\Bigl(\frac{iy-t}{i-t}
\Bigr)^{r_2-2}\, dy \= i \int_{y=1}^\infty f(iy) \, \Bigl(
\frac{iy/t-1}{i/t-1}\Bigr)^{r_2-2} \, dy\,.
\end{align*}
The integrand has the exponentially decreasing factor $f(iy)$ and
another factor that is $O(y^{r_2-2})$, uniform in $t\in \proj\RR$. So
the integral $\ph_a^-$ converges to $\ph^-$ in the supremum norm
on~$\proj\RR$.

Applying the differential operator $\frac1{1+t^2} \partial_t$ a number
of times gives an integral with a more complicated expression. We do
not need to determine these derivatives explicitly. We note that they
involve powers of $\frac1{1+t^2}$, $t$, and $\frac{iy-t}{i-t} $, which
allows us to handle the derivatives of $\ph_a^-$ and $\ph^-$ in an
analogous way. Thus we conclude that $\ph_a^- \rightarrow \ph^-$ as
$a\uparrow \infty$ in the topology of $\dsv{v_2,2-r_2}\infty$.
\end{proof}

\begin{lem}For $f_1,f_2$ as in the theorem
\be\label{limf} \bigl[ c_{f_2}^K(i,\infty), c_{f_1}(\rho-1,\rho)
\bigr]_{2-r} \= \lim_{a\uparrow\infty} \bigl[ c_{f_2}^K(i,ia),
c_{f_1}(\rho-1,\rho)
\bigr]_{2-r}\,.\ee
\end{lem}
\begin{proof}Part ii) of Theorem~\ref{thm-dual} gives continuity in the
second argument. As long as both arguments are in $\uldsv{v,2-r}\infty$
the transition to $[\iota f,\iota g]_r$ in the theorem gives continuity
in the first argument as well. To see this we use part~iv) of the
duality theorem.
\end{proof}

\rmrk{Rearranging the integrals}We proceed with
$ \bigl[ c_{f_2}^K(i,ia), c_{f_1}(\rho-1,\rho)
\bigr]_{2-r}$. The three integrals involved in it all run over compact
sets, and we can order them as pleases us. We take
\badl{cKca} \bigl[ c_{f_2}^K&(i,ia), c_{f_1}(\rho-1,\rho) \bigr]_{2-r}\\
&\= \int_{\tau_1=\rho-1}^{\rho}\, f_1(\tau_1) \, \int_{\tau_2=i}^{ia}
\overline{f_2(\tau_2)} \, \bigl[p_{\tau_2} , q_{\tau_1}\bigr]_{2-r} \,
d\tau_1\, d\bar\tau_2\,,\\
p_{\tau_2}(z) &\= \bigl( \bar\tau_2-z\bigr)^{r-2}\,,
\qquad q_{\tau_1}(z) \= \bigl( \tau_1-z\bigr)^{r-2}\,. \eadl

\rmrk{The inner integral}Part i) of the duality theorem
Theorem~\ref{thm-dual} gives
\be\label{d0} \bigl[p_{\tau_2} , q_{\tau_1}\bigr]_{2-r} \= \frac1\pi
\int_{z\in C} p_{\tau_2}^+(z) \, \s_r q_{\tau_1}^-(z)\,
\frac{dz}{(z+i)^2}\,.\ee
The closed contour $C$ in $\uhp$ encircles the path from $\tau_1$
to~$i$. We have with~\eqref{fpm}:
\be p_{\tau_2}^+(z) \= \Bigl( \frac{\bar\tau_2-z}{-i-z} \Bigr)^{r-2}\,,
\qquad q_{\tau_1}^-(z) \= \Bigl( \frac{\tau_1-z}{i-z}\Bigr)^{r-2}\,. \ee

It is convenient to proceed in disk coordinates $w=\frac{z-i}{z+i}$,
$z=i\frac{1+w}{1-w}$, $u_1=\frac{\tau_1-i}{\tau_2-i}$. With
$|u_1|<|w|<1$ we find:
\begin{align}
\nonumber
q_{\tau_1}^-(z) &\= \Bigl( \frac{u_1-w}{-w(1-u_1)} \Bigr)^{r-2} \=
(1-u_1)^{2-r}\,\bigl(1-u_1/w\bigr)^{r-2}\\
&\= (1-u_1)^{2-r} \sum_{m} \binom{r-2}m\, (-1)^m u^m\, w^{-m}\,,\\
\nonumber
\s_{2-r} q_{\tau_1}^-(z) &\= (1-u_1)^{2-r} \sum_{m}\binom{r-2}m\, (-1)^m
u^m\, \frac{m!}{(2-r)_m}\, w^{-m-1}\\
&\= (1-u_1)^{2-r} \sum_m u^m\, w^{-m-1}\,.
\end{align}
We used~\eqref{srw}. If $r\not\in \ZZ_{\geq 2}$ the sum runs over
$m\geq 0$, and over $0\leq m \leq r-2$ if $r\in \ZZ_{\geq 2}$.

We take $u_2=\frac{\tau_2-i}{\tau_2+i}$. Then
$\bar\tau_2 = - i \frac{\bar u_2+1}{\bar u_2-1}$ and
\be p_{\tau_2}^+(z) \= \Bigl( \frac{1-w\bar u_2}{1-\bar u_2}
\Bigr)^{r-2} \=
(1-\bar u_2)^{2-r} \, \sum_n \binom{r-2}n \; (-1)^n \bar u_2^n w^n
\,.\ee
Here $0\leq n \leq r-2$ if $r \in \ZZ_{\geq 2}$, and $n\geq 0$
otherwise. With use of part~iv) of the duality theorem we get
\bad\bigl[p_{\tau_2}& , q_{\tau_1}\bigr]_{2-r} \= (1-\bar
u_2^{-1})^{2-r} \,
(1-\bar u_2)^{2-r} \, \bigl(1-u_1\bar u_2\big)^{r-2}\\
&\= (2i)^{2-r}\,\bigl( \tau_1-\bar\tau_2\bigr)^{r-2}\,. \ead

Thus we find for the quantity in~\eqref{cKca}:
\be \label{cca2}\= (2i)^{2-r}\, \int_{\tau_1=\rho-1}^{\rho}
f_1(\tau_1)\int_{\tau_2=i}^{ia} \overline{f_2(\tau_2)}\, \bigl(
\tau_1-\bar\tau_2\bigr)^{r-2} \, d\bar\tau_2\, d\tau_1\,. \ee

\rmrk{Limit as $a\uparrow\infty$}The exponential decay of cusp forms
shows that the limit of the quantity in~\eqref{cca2} exists. With
\eqref{limf} we get
\bad \bigl[ c^K_{f_2}(i&,\infty), c_{f_1}(\rho-1,\rho) \bigr]_{r-2} \\
&\=(2i)^{2-r}\, \int_{\tau_1=\rho-1}^{\rho}
f_1(\tau_1)\int_{\tau_2=i}^\infty \overline{f_2(\tau_2)}\, \bigl(
\tau_1-\bar\tau_2\bigr)^{r-2} \, d\bar\tau_2\, d\tau_1\\
&\=
(2i)^{2-r} \,\int_{\tau_1=\rho-1}^{\rho}\int_{\tau_2=i}^\infty
\om(\tau_1,\tau_2)\,, \ead
with the differential form on $\uhp\times\uhp$
 \be \om(\tau_1,\tau_2) = f_1(\tau_1) \,
 \overline{f_2(\tau_2)}\,(\tau_1-\bar\tau_2)^{r-2}\, d\bar\tau_2\,
 d\tau_1\,.\ee
 This differential form is invariant for the diagonal action of $\Gm$.
 Hence we have also
 \be \label{dfe} \bigl[ c^K_{f_2}(i,\infty), c_{f_1}(\rho-1,\rho)
\bigr]_{r-2} \= -(2i)^{2-r} \int_{\tau_1=i}^\rho
\int_{\tau_2=0}^{\infty} \om(\tau_1,\tau_2) \,. \ee

\rmrk{Partial integration}We proceed as in \cite[\S7.2]{Ha}. The
function
\be F_2(\tau_1) \= \int_{\tau_2=\tau_1}^\infty \overline{f_2(\tau_2)}\,
(\tau_1-\bar\tau_2)^{r-2}\, d\bar\tau_2\qquad(\tau_1\in \uhp)\,,\ee
is not holomorphic, but satisfies
\be\label{F2harm}
\partial_{\bar\tau_1} F_2(\tau_1) \= - \overline{f_2(\tau_1)} \bigl(
\tau_1-\bar\tau_1)^{r-2} \= -(2i)^{r-2}\, (\im \tau_1)^{r-2} \,
\overline{f_2(\tau_1)}\,, \ee
and has the following transformation behavior under $\gm\in \Gm$:
\be\label{F2trf} F_2|_{v^{-1},2-r} \gm (z) \= F_2(z) + c^K_{f_2}(\infty,
\gm^{-1}\infty)(z)\,. \ee

We use that
\begin{align*}
d \Bigl( f_1(\tau) \, F_2(\tau)\, d\tau \Bigr) &\= - f_1(\tau) \bigl(
\partial_{\bar\tau} F_2(\tau) \bigr) \, d\tau \, d\bar\tau\\
&\= -(2i)^{r-1} \, f_1(x+iy) \, \overline{f_2(x+iy)} \, y^{r-2}\,
dx\,dy\,,
\end{align*}
to compute the Petersson scalar product on the fundamental domain
$\fd_1= R \cup TL$ (not the standard fundamental domain), with $R$ and
$L$ as in Figure~\ref{fig-tessf} on p~\pageref{fig-tessf}. We use
Stoke's theorem to get:
\begin{align*}
\bigl( f_1,f_2)_r&\=-(2i)^{1-r} \int_{\partial\fd_1} f_1(\tau_1) \,
F_2(\tau_1)\, d\tau_1\\
&\= -(2i)^{1-r}\, \biggl( \int_{i+1}^\infty - \int_{i}^\infty +
\int_{i}^\rho - \int_{i+1}^\rho\biggr)\, f_1(\tau_1) \, F_2(\tau_1)\,
d\tau_1 \,.
\end{align*}
The transformation behavior in~\eqref{F2trf} implies that the first two
integrals cancel each other, and that the remaining integrals give
\begin{align*} -(2i)^{1-r}& \int_{\tau_1=i}^\rho f_1(\tau_1)\,
c^K_{f_2}(\infty,0)(\tau_1)\, d\tau_1\\
&\= (2i)^{1-r}\, \int_{\tau_1=i}^\rho \int_{\tau_2=0}^\infty
\om(\tau_1,\tau_2)\,.
\end{align*}
Comparison with~\eqref{dfe} completes the proof of  Theorem~\ref{thm-Hab}.
\end{proof}

\section{Coinvariants of polynomial functions}\label{sect-clp}
The case cEE in the discussion in Subsection~\ref{sect-amf} 
  considers   the
product
\[\cp(f_1,f_2) \= c_{f_1}(\rho-1,\rho) \cdot c_{f_2}(i,\infty)\]
for two cusp forms $f_1\in \cf{r_1}(v_1)$, $f_f\in \cf{r_2}(v_2)$ with
weights $r_j\in \ZZ_{\geq 2}$ and the $v_j$ multiplier systems
corresponding to~$r_j$. It represents a coinvariant in the finite
dimensional module $\dsv{v_1v_2, 4-r_1-r_2}\pol$. If we take trivial
multiplier systems this is the classical situation of the cup product
of two Eichler cohomology classes.

It turns out that in this classical context the cup product of two
Eichler cocycles is uninteresting.

\begin{prop}\label{prop-clci}For a weight $r\in \ZZ_{\geq 2}$ and a
corresponding multiplier system $v$
\be \bigl(\dsv{v,2-r}\pol\bigr)_\Gm \= \begin{cases}
\CC &\text{ if }r=2\text{ and }v=1\,,\\
\{0\}&\text{ otherwise}\,.
\end{cases}\ee
\end{prop}
\begin{proof}For integral weights multiplier systems are characters. For
any polynomial $p$ of the form
$p(t) = t^a + \text{lower degree terms}$, we have
\[ p|^-_{v,2-r}(1-T) \= \bigl( 1-v(T) \bigr) \, t^a + \text{lower degree
terms}\,.\]
So if $v(T) \neq 1$ we obtain by induction on the degree that
$ \dsv{v,2-r}\pol|^-_{v,2-r}(1-T)
= \dsv{v,2-r}\pol $.

The value on $T$ determines the multiplier system, hence we are left
with the case $v=1$. For $p$ of degree $a$ as above we have
\[ p|^-_{1,2-r}(1-T) = - a t^{a-1} + \text{lower degree terms}\,.\]
 By induction on $a$ we conclude that
 $t^a \in \dsv{1,2-r}\pol|_{1,2-r}^-(1-T)$ for $a=0,\ldots,r-3$. Finally
 we note that
\begin{align*} t^{r-2}|^-_{1,2-r} (1- S) &\=t^{r-2}- v(S)^{-1}\,
t^{r-2}\,
(-1/t)^{r-2}\\
& \= t^{r-2}-1\cdot (-1)^{r-2}\in t^{r-2} +
\dsv{1,2-r}\pol|^-_{1,2-r}(1-T)\,. \qedhere
\end{align*}
\end{proof}

This implies that in case cEE the polynomial $\cp(f_1,f_2)$ indicated
above represents the trivial coinvariant in $\dsv{v_1,v_2}\pol$ if
$r_1>2 $ or $r_2>2$, and also if $r_1=r_2=2$ and $v_1v_2\neq 1$.

For the modular group the sole multiplier system $v$ for which
$\cf 2(v)\neq\{0\}$ is determined by $v(T)= e^{\pi i/3}$. (The
corresponding space of cusp forms is spanned by $\eta^4$.) This
multiplier system does not satisfy $v^2=1$, and hence it is
understandable that in the classical context the cup product of two
Eichler cocycles of modular   forms   is uninteresting.

\section{Coinvariants associated to two Eichler cocycles}\label{sect-EE}
In the case cEE in the previous section it turned out that for modular
forms
(on the full modular group $\SL_2(\ZZ)$) with integral weight at least
$2$ the cup product leads to the trivial coinvariant. In the case EE in
\S\ref{sect-amf} we consider a modular form $f_1\in A_{r_1}(v_1)$ and
$f_2\in \cf{r_2}(v_2)$ and form the coinvariant in
$\dsv{v_1v_2,4-r_1-r_2}\infty$ represented by the following product of
two values of Eichler cocycles
\[ \cp(f_1,f_2) \= c_{f_1}(\rho-1,\rho) \cdot c_{f_2}(i,\infty)\,.\]
Under the assumption $r_2\not\in \ZZ_{\geq 2}$ this coinvariant
represents the image of the cup product
\[ c_{f_1} \cup c_{f_2} \in H^2_p\bigl( \Gm;
\dsv{v_1,2-r_1}\infty\otimes \dsv{v_2,2-r_2}\infty\bigr)
\;\cong\; \bigl( \dsv{v_1,2-r_1}\infty\otimes
\dsv{v_2,2-r_2}\infty\bigr)_\Gm \]
under the map in cohomology corresponding to the linear map

\[ \dsv{v_1,2-r_1}\infty\otimes \dsv{v_2,2-r_2}\infty \rightarrow
\dsv{v_1v_2,4-r_1-r_2}\infty \]
induced by $v\otimes w \mapsto v\,w$.

To simplify the formulas we use the multiplier system
$v_3=v_1^{-1}v_2^{-1}$ corresponding to the weight $r_3=4-r_1-r_2$.
Since $\Gm$ is generated by $S=\matr0{-1}10$ and $T=\matc1101$, the
space of coinvariants is the quotient of  the 
infinite-dimensional module $\dsv{v_3^{-1},r_3}\infty $ by the
submodule
\[ \dsv{v_3^{-1},r_3}\infty |^-_{v_3^{-1},r_3} (1-S)
+ \dsv{v_3^{-1},r_3}\infty |^-_{v_3^{-1},r_3}(1-T)\,.\]
It is hard to understand this submodule. Even the question whether
$\cp(f_1,f_2)$ represents the trivial coinvariant is hard to answer.

We use the fact that each $\Gm$-invariant linear form $\bt$ on
$\dsv{v_3^{-1},r_3}\infty $ is trivial on the submodule, and induces a
linear form on $\bigl(\dsv{v_3^{-1},r_3}\infty \bigr)_\Gm$. The duality
theorem Theorem~\ref{thm-dual} has the following consequence:

\begin{cor}\label{cor-dual}Let $r_3\geq 0$, and let
$c\in \dsv{v_3^{-1},r_3}\infty$ represent a coinvariant
$[c]\in \bigl(\dsv{v_3^{-1},r_3}\infty\bigr)_\Gm$. If there exists a
entire modular form $h\in M_{r_3}(v_3)$ for which
$\bigl[ h,c\bigr]_{r_3}\neq0$, then the coinvariant $[c]$ is
non-trivial.
\end{cor}
\begin{proof}The space of entire modular forms $M_{r_3}(v_3)$ is
characterized by $\Gm$-invariance and polynomial growth at the cusps.
Polynomial growth near the boundary implies polynomial growth at the
cusps, so
$\bigl(\udsv{v_3,r_3}{-\infty}\bigr)^\Gm\subset M_{r_3}(v_3)$. With use
of the Fourier expansion one checks that elements of $M_{r_3}(v_3)$
have polynomial growth near $\RR$. So $M_{r_3}(v_3)$ is equal to
$\bigl( \dsv{v_3^{-1},r_3}{-\infty}\bigr)^\Gm$.

Suppose that $c$ represents the trivial coinvariant. Then
$c\in \dsv{v_3^{-1},r_3}\infty|^-_{v_3^{-1},r_3}(1-\nobreak S) + \dsv{v_3^{-1},r_3}\infty|^-_{v_3^{-1},r_3}(1-\nobreak T)$,
and any $\Gm$-invariant linear form $\al$ on
$ \dsv{v_3^{-1},r_3}\infty$ satisfies $\al(c)=0$. In particular this
would mean $[h,c]_{r_3}=0$. Hence $c$ in the corollary represents a
non-trivial coinvariant.
\end{proof}

\rmrk{Remark}We do not know whether
$ \dsv{v^{-1},r}\infty|^-_{v^{-1},r}(1-S) + \dsv{v^{-1},r}\infty|^-_{v^{-1},r}(1-T)$
is closed in $\dsv{v^{-1},r}\infty$ for the natural topology on
$\dsv{v^{-1},r}\infty$. So if $[h,c]_r=0$ for all $h\in M_r(v)$, there
still might be a non-continuous $\Gm$-invariant linear form $\al$ for
which $\al(c)\neq 0$. The condition in the corollary is sufficient but
not necessary for non-triviality of~$[c]$.

\rmrk{Trilinear form}We proceed with $r_1,r_2,r_3\in \RR$ and
corresponding multiplier systems $v_1,v_2,v_3$ satisfying
\be\label{rv-assumpt} r_1+r_2+r_3\=4\, \qquad r_2>0\,,\qquad r_3\geq
0\,\qquad v_1v_2v_3\=1 \ee
and consider the trilinear form
\be\label{T-def}\T(f_1,f_2,f_3)\= \bigl[f_3,\cp(f_1,f_2)\bigr]_{r_3}\ee
on $A_{r_1}(v_1) \times \cf{r_2}(v_2) \times M_{r_3}(v_3)$.

Our aim in the next section is to reformulate $\T(t_1,f_2,f_3)$ more
explicitly (under stronger conditions than \eqref{rv-assumpt}) such
that it can be computed numerically.

The trilinear form $\T$ makes sense without the condition
$r_2\not\in \ZZ_{\geq 2}$. Under this additional condition we know that
if $\T(f_1,f_2,f_3)\neq0$ for some choice of $(f_1,f_2,f_3)$, then
$\cp(f_1,f_2)$ represents a non-trivial coinvariant.

\subsection{Triviality over larger modules}\label{sect-class}
 In this subsection we give ourselves the task to
describe the $2$-cocycle corresponding to $\cp(f_1,f_2)$ as a
$2$-coboundary over some larger $\Gm$-module, or equivalently, to write
\[ \cp(f_1,f_2) \= A_1|(T-1)+A_2|(S-1)\]
with some $A_1$ and $A_2$ in a larger $\Gm$-module. The following result
establishes this for two larger modules.
\begin{prop}\label{prop-triv}
Let $f_1\in A_{r_1}(v_1)$, $f_2\in \cf{r_2}(v_2)$ with $r_1\in\RR$ and
$r_2>0$. As above we use $r_3=4-r_1-r_2$, and multiplier systems
satisfying $v_3=v_1^{-1}v_2^{-1}$.
\begin{enumerate}
\item[i)]There are fairly explicit functions $A_1, A_2$ in the module of
all real-analytic functions on~$\lhp$ with the action
$|^-_{v_3^{-1},r_3}$ such that
\be \cp(f_1,f_2) \= A_1|_{v_3^{-1},r_3}^- (T-1)
+ A_2|_{v_3^{-1},r_3}^- (S-1))\,.\ee
\item[ii)]There are $B_1,B_2\in \dsv{v_3^{-1},r_3}{-\om}$ such that
\be \cp(f_1,f_2) \= B_1|_{v_3^{-1},r_3}^- (T-1)
+ B_2|_{v_3^{-1},r_3}^- (S-1)\,.\ee
\end{enumerate}
\end{prop}

The proof takes the remainder of this subsection. It depends on various
other results, which may be considered interesting independently. In
the course of the proof the meaning of \emph{fairly explicit} will
become clear.

\begin{lem}\label{lem-pot}Let $V$ be a linear space of functions on the
lower half-plane containing the holomorphic functions and stable under
multiplication by holomorphic functions.

Let $c_1\in Z^1\bigl( \Gm; \dsv{v_2,2-r_2}\infty)$ and
$c_1\in Z^1_p\bigl( \Gm; \dsv{v_2,2-r_2}\infty)$. If there exist
elements $q\in V$ such that
\be \label{q-cond}
q|_{v_2,2-r_2}^-T \= q|_{v_2,2-r_2}^-S\,,\quad\text{ and }\quad
q|_{v_2,2-r_2}^-(S-1) \= c_1(\rho-1,\rho)\,,\ee
then
\badl{genexpr}
 &c_1(\rho-1,\rho)\, c_2(i,\infty)\\
&\quad\=\Bigl( q \, c_2(\rho,\infty\Bigr) \bigm|^-_{v_3^{-1},r_3}(T-1)+
\Bigl(\bigl(q|_{v_1,2-r_1}^- S\bigr)\,
c_2(\rho-1,i)\Bigr)\bigm|^-_{v_3^{-1},r_3}(S-1)\,. \eadl
\end{lem}
\rmrk{Remark}We consider \eqref{genexpr} to be `fairly explicit' in the
data $c_1, c_2$, and $q$.
\begin{proof}By the following computation, starting at the left hand
side in~\eqref{genexpr}. To save space we denote by $|$ the actions
$|_{v_1,2-r_1}$ and $|^-_{v_2,2-r_2}$ on the separate factors
\begin{align*}
&\= \bigl( q|T)\, c_2(\rho-1,\infty) - q\, c_2(\rho,\infty) + q \,
c_2(\rho,i) - (q|S) \, c_2(\rho-1,i)
\displaybreak[0]\\
&\= (q|S)\, \bigl( c_2(\rho-1,\infty)-c_2(\rho-1,i) \bigr) - q \,\bigl(
c_2(\rho,\infty)- q\, c_2(\rho,i) \bigr)\\
&\=\Bigl( q|S-q) \, c_2(i,\infty) \= c_1(\rho-1,\rho)\,
c_2(i,\infty)\,.\qedhere
\end{align*}
\end{proof}
\rmrk{Remark}We found the relation by using the fact the the linear map
in \eqref{H12cp} on cohomology classes. So the cup product
$dp \cup c_2$ should be a coboundary. Evaluating $p \otimes c_2$ on
$\delta_1 \partial\left(L+R\right)$ leads to \eqref{genexpr}.

If we try to work with the cup product $c_1 \cup dp_2$ the function
$p_2$ would have to have values in a module in which there are non
non-trivial invariants for the element $T\in \Gm$. Cocycles of modular
forms do not become trivial in modules satisfying this condition.
\begin{proof}[Proof of part~i) of Proposition~\ref{prop-triv}.] The
space $V^{\mathrm{an}}$ of all real-analytic functions on~$\lhp$
satisfies the condition in Lemma~\ref{lem-pot}. We put for $t\in \lhp$
\be Q_{f_1}(t) \= \int_{\tau=\rho}^{\bar t}\, f_1(\tau)\, \bigl( \tau-t
\bigr)^{r_1-2}\,d\tau\,.\ee
The presence of $\bar t$ as limit of integration makes $Q_{f_1}$
non-holomorphic. It is real analytic. A direct computation shows that
for all $\gm\in \Gm$
\be Q_{f_1} |^-_{v_1,2-r_1}(\gm-1) \= \int_{\tau=\gm^{-1}\rho}^\rho
f_2(\tau) \, (\tau-t)^{r_1-2}\, d\tau\,. \ee
So the group cocycle
$\gm\mapsto \ps_\gm^\rho= c_{f_1}(\gm^{-1}\rho,\rho)$ associated to
$f_1$ with base point $\rho$ becomes a coboundary in the module
$\bigl(  V^{\mathrm{an}}, |^-_{v_2,2-r} \bigr)$. It satisfies
$\ps_{TS}^\rho=0$ since $ST ^{-1}\rho=\rho$. So
$Q_{f_1}|^-_{v_2,2-r_2}TS=Q_{f_1}$. We can take $q= Q_{f_1}$ in
Lemma~\ref{lem-pot}.

The construction of $Q_{f_1}$ may also be considered \emph{fairly
explicit}.
\end{proof}

\begin{prop}\label{prop-X}For all $r_1\in\RR$ and corresponding
multiplier system the space of automorphic functions has a
decomposition
\be A_{r_1}(v_1) \= \cf{r_1}(v_1) \oplus X_r(v_1)\,,\ee
  where $X_{r_1}(v_1)$ is the space of 
$f_1\in A_{r_1}(v_1)$ for which  
there is an element
$q_{f_1}\in \dsv{v_1,2-r_1}{-\infty}$ such that
\be c_{f_1}(\gm^{-1}\rho,\rho) \= q_{f_1}|^-
_{v_1,2-r_1}(\gm-1)\qquad\text{for all }\gm\in \Gm\,.\ee
\end{prop}
\begin{proof}This is a consequence of the theorem of Knopp and Mawi
\cite{KM10} which gives a bijection
\be \cf{r_1}(v_1) \longrightarrow H^1(\Gm; \dsv{v_1,2-r_1}{-\infty}
)\,,\ee
given by assigning to a cusp form $f_1 $ the cohomology class
represented by the group cocycle
$\gm\mapsto \ps^\infty= c_{f_1}(\gm^{-1}\infty,\infty)$. Actually, they
state the result with Knopp cocycles $\iota \ps^\infty$.

The natural maps associated to extension of modules give linear maps
\be\label{AHHH}
A_{r_1}(v_1) \rightarrow H^1(\Gm; \dsv{v_1,2-r_1}\om)
 \rightarrow H^1(\Gm; \dsv{v_1,2-r_1}\infty)
  \rightarrow H^1(\Gm; \dsv{v_1,2-r_1}{-\infty})\ee
  sending $f_1$ to the cohomology class $\bigl[\ps^\rho]$. For cusp
  forms $f_1$ the classes of $\ps^\rho$ and $\ps^\infty$ coincide in the
  last two of the modules in~\eqref{AHHH}. We define $X_{r_1}(v_1)$ as
  the kernel of
 $A_{r_1}(v_1) \rightarrow  H^1(\Gm; \dsv{v_1,2-r_1}{-\infty})$. The
 theorem of Knopp and Mawi implies the statements in the proposition.
\end{proof}
\rmrk{Remark}In principle the construction of $q_{f_1}$ might be traced
by analyzing the proofs in~\cite{Kn74} and~\cite{KM10}. We would not
call the result \emph{fairly explicit}.

\begin{prop}\label{prop-hl}Let $f_1\in A_{r_1}(v_1)$ have at most
exponential growth $\oh\bigl( e^{A\im z}\bigr)
$ with some $A\in R$ at the cusp. Then there are elements
$q\in \dsv{v_1,2-r_1}{-\om}$ such that
\be c_{f_1}(\gm^{-1}\rho,\rho) \= q_{f_1}|^-
_{v_1,2-r_1}(\gm-1)\qquad\text{for all }\gm\in \Gm\,.\ee
\end{prop}
\begin{proof}Theorem~C in \cite{BCD} implies that such $q$ exist if the
automorphic form $f_1$ has a $(2-\nobreak r)$-harmonic lift. For
exponentially growing modular forms the existence of such harmonic
lifts follows from Theorem~1.1 in~\cite{Br14}.\end{proof}

\rmrk{Remark}The proof of the existence of harmonic lifts is highly
non-explicit. Here we need it only for cusp forms. Then we may work
with the cocycle $\ps^\infty_{f_1}$. The corresponding functions
$q^\infty_{f_1}$ are mock modular forms with shadow (a multiple of)
$f_1$. The difference $q^\infty_{f_1} - q_{f_1}$ be made more or less
explicit.

\begin{proof}[Proof of part~ii) of Proposition~\ref{prop-triv}.]The
function $q$ in Lemma~\ref{lem-pot} is provided by
Proposition~\ref{prop-X} for $f_1\in X_{r-1}(v_1)$. For cusp forms we
use Proposition~\ref{prop-hl}.
\end{proof}
\section{Triple integral}\label{sect-tri}
Our aim in this section is to prove Theorem~\ref{thm-tri}, which gives
an expression as a threefold iterated integral for the trilinear form
 $\T(f_1,f_2,f_3)$ introduced in~\eqref{T-def}. We reformulate the
 triple integral in Theorem~\ref{thm-tri} in
Proposition~\ref{prop-triFc}. This formulation is sufficiently explicit
to allow numerical computations, in Section~\ref{sect-num}, that
suggest that for many choices   of the weights   the 
trilinear form  $(f_1,f_2,f_3)\mapsto \T(f_1,f_2,f_3)$   is
non-zero.

Throughout this section we consider three modular forms
$f_1\in A_{r_1}(v_1)$, $f_2\in \cf{r-2}(v_2)$, $f_3\in M_{r_3}(v_3)$,
and assume that the weights $r_j$ and corresponding multiplier systems
$v_j$ satisfy the conditions in~\eqref{rv-assumpt}.

From   Subsection~\ref{sect-trip}   on we work under 
the stronger conditions
\be\label{str-cond}r_2<2\,,\quad 0<r_2<2\,,\quad r_1+r_2+r_3=4\,,\quad
v_1v_2v_3=1\,.\ee
The triple integral as stated in Theorem~\ref{thm-tri} is absolutely
convergent if $r_1,r_2<2$. We expect that analytic continuation in
$(r_1,r_2)$ is possible, but do not pursue this here.

\subsection{Truncation and fourfold integral}We start without the
additional assumption $r_1,r_2<2$.

\rmrk{Truncation}Like in the proof of Theorem~\ref{thm-Hab} we
approximate $c_{f_2}(i,\infty)$ by $c_{f_2}(i,ia)$ with $a>0$, and put
\be\label{Tadef}
 \T_a(f_1,f_2,f_3) \= \Bigl[ f_3, c_{\!f_1}(\rho-1,\rho)\cdot c_{\!f_2}(
 i,ia)
\Bigr]_{r_3}\,.\ee

\begin{lem}\label{lem-Tlimit}Under the assumptions \eqref{rv-assumpt}
\be \T(f_1,f_2,f_3) \= \lim_{a\rightarrow\infty} \T_a(f_1,f_2,f_3)
,.\ee
\end{lem}
\begin{proof}Lemma~\ref{lem-a} states that $c_{f_2}(i,ia)$ approximates
$c_{f_2}(i,\infty)$ as $a\uparrow \infty$ in the natural topology on
$\dsv{v_2,2-r_2}\infty$. For fixed
$\ps = c_{f_1}(\rho-\nobreak 1,\rho)$ the multiplication
$\ph \mapsto \ps\,\ph$ is a continuous map
$\dsv{v_2,2-r_2}\infty \rightarrow \dsv{v_1v_2,4-r_1-r_2}\infty$. Since
the duality $[h,\ph]_{r_3}$ is continuous in $\ph$ for fixed $h$ the
lemma follows.
\end{proof}

\rmrk{Fourfold integral}The quantity $\T_a(f_1,f_2,f_3)$ has the
advantage that it has an expression involving four integrals over
compact sets: the two Eichler integrals from $\rho-1$ to $\rho$ and
from $i$ to $ia$, and the two integrals over compact cycles in the
duality theorem Theorem~\ref{thm-dual} and in the definition of
$\s_rf$. This is quite complicated. However, the compact domains of
integration allow us to choose any order of integration that suits us.
\begin{prop}\label{prop-4i}Under the assumptions \eqref{rv-assumpt} and
$r_3>0$:
\bad \T(f_1,&f_2,f_3) \= \lim_{a\rightarrow\infty}
\int_{\tau_1=\rho-1}^\rho f_1(\tau_1)
\int_{\tau_2=i}^{ia} f_2(\tau_2) \, \frac1\pi \int_{z\in C_2(a)}
f_3(z)\,
(-z-i)^{r_3}\\
&\quad\hbox{} \cdot
\frac1\pi \int_{\tau\in C_1(a)}\,
\hypg21\biggl(1,1;r_3;\frac{(\tau-i)(z+i)}{(\tau+i)(z-i)}\biggr)\,
\Bigl( \frac{\tau_1-\tau}{i-\tau}\Bigr)^{r_1-2}\, \Bigl(
 \frac{\tau_2-\tau)}{i-\tau}\Bigr)^{r_2-2}\\
 &\quad\hbox{} \cdot
\frac{d\tau}{\tau^2+1}\; \frac{dz}{z^2+1}\; d\tau_2 \; d\tau_1\,. \ead
We take $C_2(a)$ and $C_1(a)$ as positively oriented circles
$\bigl|\frac{z-i}{z+i}\bigr|=c_2(a)$ and
$\bigl|\frac{\tau-i}{\tau+i}\bigr|=c_1(a)$ with $c_1(a)< c_2(a)<1$ such
that $C_1(a)$ encircles paths from $\rho-1$ to $\rho$ and from $i$ to
$ia$. See Figure~\ref{fig-4fi}.
\end{prop}
\begin{figure}[tp]
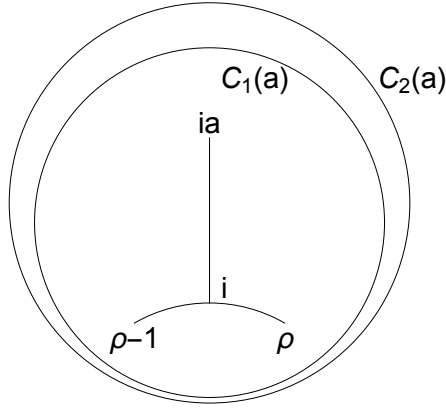

\begin{center}\grf7{cpmc-4fi}\end{center}
\caption{Paths of integration in Proposition~\ref{prop-4i}.}
\label{fig-4fi}
\end{figure}
\begin{proof}We unravel the definitions in Theorem~\ref{thm-dual}, and
in \eqref{sigrdef}, \eqref{fpm} and \eqref{Eidef}, and combine them as
the limit of a fourfold integral. We use that $r_1+r_2=4-r_3$. For
$\tau\in C_1(a)$ we can combine $(\tau_j-\nobreak \tau)^{r_j-2}$ and
$(i-\nobreak\tau)^{2-r_j}$ to get
$\bigl( \frac{\tau_j-\tau}{i-\tau}\bigr)^{r_j-2}$ as a holomorphic
function on $\proj\CC$ minus a path from $\tau_j$ to~$i$. Actually, in
\eqref{sigrdef} we need the following form of the Eichler integrals
\[ c_{f_j}(x_1,x_2)^-(\tau) = (i-\tau)^{2-r_j} \,
c_{f_j}(x_1,x_2;\tau)\,,\]
since these functions are holomorphic on a neighborhood of
$\lhp\cup\proj\RR$ in~$\proj\CC$. They are directly given by
\be c_{f_j}(x_1,x_2)^-(\tau) \= \int_{\tau=x_1}^{x_2} \, f_j(\tau_j)
\,\Bigl(\frac{\tau_j-\tau}{i-\tau}\Bigr)^{r_j-2}\, d\tau\,.\ee
We need $r_3>0$, since the hypergeometric function in formula
\eqref{sigrdef} is not defined at $r_3=0$.
\end{proof}

\subsection{Triple integral}\label{sect-trip}
In this subsection we prove Theorem~\ref{thm-tri}. To the assumptions
\eqref{rv-assumpt} we add in~\eqref{str-cond} the conditions $r_1<2$,
$r_2<2$. These additional conditions restrict the triples
$f_1,f_2,f_3)$ considerably. We have $r_3>0$, and moreover $f_2$ can
only be a multiple of $\eta^{2r_2}$, hence $v_2=v[r_2]$.

The plan of the proof is to modify the fourfold integral in
Proposition~\ref{prop-4i}. We start with the two inner integrals, which
describe $\bigl[ f_3, \kp(\tau_1,\tau_2;\cdot) \bigr]_{r_3}$, with the
product of modified Eichler kernels
\be\label{kpd1} \kp(\tau_1,\tau_2;\tau) \=
\Bigl(\frac{\tau_1-\tau}{i-\tau}\Bigr)^{r_1-2}\,
\Bigl(\frac{\tau_2-\tau}{i-\tau}\Bigr)^{r_2-2}\,.\ee
Lemma~\ref{lem-gtf}, derived in the context of the universal covering
group of $\SL_2(\RR)$, gives a simpler expression for the two innermost
integrals.

The main work is the transformation is the simplification of this
expression. We use Kummer relation for the hypergeometric function, and
go over from a closed contour to a segment in $\uhp$ as the path of
integration.

Combining the resulting integral with the integrals over $\tau_1$ and
$\tau_2$ finishes the proof.\medskip

\rmrk{Standing assumptions}In stating lemmas we work with the standard
assumptions for the $f_j$, the $r_j$, and the $v_j$. From
Lemma~\ref{lem-ti2} onwards we assume $r_1<2$, $r_2<2$.

\begin{lem}\label{lem-ti1}Let $\tau_1\neq\tau_2$. Then
\badl{ti1} \bigl[ f_3&, \kp(\tau_1,\tau_2;\cdot)\bigr]_{r_3}
\= \frac{(2i)^{2-r_2}}\pi \bigl|c\tau_1+d\bigr|^{2r_1-2}\,
(\tau_2-\bar\tau_1)^{r_2-2} \\
&\qquad\hbox{} \cdot
\int_{z\in C} f_3(z)\, (-i-z)^{r_3}\, \Bigl(
\frac{z+i}{z-\bar\tau_1}\Bigr)^{1-r_3}\\
&\qquad\hbox{} \cdot
\hypg21\biggl( 1,2-r_2;r_3;
  \frac{(\tau_2-\tau_1)(z-\bar\tau_1}{(\tau_2-\bar\tau_1)(z-\tau_1)}
  \biggr)\, \frac{dz}{(z-\tau_1)(z+i)}\,, \eadl
where $C$ is a positively oriented curve in $\uhp$ encircling all
singularities of the integrand.The matrix $g=\matc abcd\in\SL_2(\RR)$
satisfies $g\tau_1 = i$ and $g\tau_2\in i(1,\infty)$, and
$\arg(ci+\nobreak d) \in[0,\pi)$.
\end{lem}
\begin{proof}This   is   Lemma~\ref{lem-gtf} with
$h^+(z) = f_3(z) \,(-i-z)^{r_3}$.
\end{proof}
Relation \eqref{ti1} is complicated, since it depends directly on
$\tau_1$ and $\tau_2$, and indirectly via the matrix~$g$. In the
following result we use the lower row $(c,d)$ of $g$ again.

\begin{lem}\label{lem-ti2}If $r_1,r_2<2$ the integral in~\eqref{ti1} is
equal to:
\bad - \pi &\, 2^{2-r_1} \, e^{-\pi i (r_1/2+r_3)}\,
\Bt(2-r_1,2-r_2)^{-1}\,
(\tau_2-\bar\tau_1)^{2-r_2}\, \bigl|c\tau_1+d\big|^{2-2r_1}\\
 &\qquad\hbox{} \cdot
\int_{z=\tau_1}^{\tau_2} f_3(z)\, \Bigl( \frac{z-\tau_1}{\tau_2-\tau_1}
\Bigr)^{1-r_2}\, \Bigl( \frac{\tau_2-z}{\tau_2-\tau_1}\Bigr)^{1-r_1}\,
\frac{dz}{\tau_2-\tau_1}\,. \ead
The path of integration can be chosen along the geodesic segment
$s_{\tau_1,\tau_2}$ from $\tau_1$ to $\tau_2$, or be deformed in such a
way that it does not cross the geodesic $\ell_{\tau_1,\tau_2}$ through
$\tau_1$ and $\tau_2$ in points of
$\ell_{\tau_1,\tau_2}\setminus s_{\tau_1,\tau_2}$.
\end{lem}
We use the beta-function $\Bt(a,b)=\Gf(a)\,\Gf(b)/\Gf(a+b)$.
\begin{proof} We rewrite for $z\in \uhp$
\[ (-i-z)^{r_3}\, \bigl( \frac{z+i}{z-\bar \tau_1}\Bigr)^{1-r_3} \=
e^{-\pi i r_3}\, (z+i)\,(z-\bar\tau_1)^{r_3-1} \,.\]
In the resulting expression for the integral in~\eqref{ti1} we go over
to the disk coordinate $w=\frac{z-\tau_1}{z-\bar\tau_1}$,
$z=\frac{\tau_1-w\bar\tau_1}{1-w}$, and put
$q=\frac{\tau_2-\tau_1}{\tau_2-\bar\tau_1}$. We choose the curve $C$
such that $|w|=q_1$ with $|q|< q_1<1$. We write $f_3[w]=f_3(z)$. The
integral is equal to
\begin{align}
\nonumber
&\= e^{-\pi i r_3} (\tau_1-\bar\tau_1)^{r_3-1} \\
\nonumber
&\qquad\hbox{} \cdot \int_{|w|=q_1} f_3[w]\, \hypg21\bigl(
1,2-r_2;r_3;q/w\bigr)\, (1-w)^{-r_3}\, \frac{dw}w\\
\label{ui}
&\=e^{-\pi i r_3} (\tau_1-\bar\tau_1)^{r_3-1} \\
\nonumber
&\qquad\hbox{} \cdot
 \int_{|u|=q_1/|q|} f_3\bigl[ qu\bigr] \, \hypg21\bigl(
1,2-r_2;r_3;u^{-1}\bigr)\,
(1-qu)^{-r_3} \, \frac{du}u\,.
\end{align}

We use Kummer relations (1), (7) \S6.5 and (9), (13), (17), (21) \S6.4
in \cite[\S6.5]{Lu} to get, under the additions conditions
$r_1,r_2\neq 1$:
\begin{align}
\nonumber
\hypg21\bigl( 1&,2-r_2;r_3;z\bigr) \=
\frac{\Gf(r_2-1)\,\Gf(r_3)}{\Gf(2-r_1)}\,
 (-z)^{r_2-2}(1-1/z)^{1-r_1}\\
 \label{KRi}
&\qquad\hbox{}
+\frac{r_3-1}{r_2-1} \, z ^{-1}\, \hypg21\bigl(1,2-r_3;r_2;1/z)\\
\nonumber
&\=\frac{\Gf(r_1-1)\,\Gf(r_3)}{\Gf(2-r_2)}\, z^{1-r_3}\,(1-z)^{1-r_1}\\
\label{KR1}
&\qquad\hbox{}
+ \frac{r_3-1}{1-r_1}\, \hypg21\bigl(1,2-r_2;r_1;1-z)\,.
\end{align}
We use that in both cases one hypergeometric function specializes to a
simpler expression.

Relation \eqref{KR1} is valid for $z\in (0,1)$ and extends to a relation
between holomorphic extensions on
$\CC \setminus \left( -\infty,0]\cup[1,\infty) \right)$. In relation
\eqref{KRi} there are no $z$ for which both hypergeometric series are
in their domain of absolute convergence. The relations hold for the
holomorphic extension for $z\in (-1,0)$, and extend to
$\CC\setminus [0,\infty)$.

The two singularities of the integrand in \eqref{ui} are determined by
the behavior of the hypergeometric function at $u=0$ and at $u=1$. The
Kummer relations show that
\be \hypg21\bigl(1,2-r_2;r_3;1/u\bigr) \ll
\begin{cases} |u|^{\min(1,2-r_2)} &\text{ as } u \rightarrow 0\,,\\
|u-1|^{\min(0,1-r_1)}& \text{ as } u\rightarrow 1\,.
\end{cases}
\ee
Together with $u^{-1}$ in $\frac{du}u$ this implies that the original
path of integration $|u|=q_1/|q|$ can be moved closely to the interval
$[0,1]$. We have to determine the contributions of the limits of the
integrals on both sides of the interval.
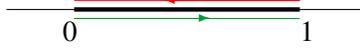
\[ \setlength\unitlength{3cm}
\begin{picture}(1.4,.2)(-.2,-.1)
\put(-.3,0){\line(1,0){1.6}}
\put(1,-.14){$1$}
\put(-.05,-.14){$0$} \red
\put(0,.04){\line(1,0){1}}
\put(1,.04){\vector(-1,0){.6}}
\endred \green
\put(0,-.04){\line(1,0){1}}
\put(0,-.04){\vector(1,0){.6}}\endgreen
\thicklines
\put(0,.005){\line(1,0){1}}
\put(0,-.005){\line(1,0){1}}
\end{picture}
\]

We have to consider the integrand in \eqref{ui} at points $u=x\pm i \e$
with $x\in (0,1)$ and $\e\downarrow 0$. The factors $f_3[qu]$ and
$(1-\nobreak qu)^{-r_3}$ are holomorphic at the points $u\in (0,1)$.

We use the Kummer relation~\eqref{KR1}. The points $1-1/x$ are in
$(-\infty,0)$. So the hypergeometric function
$\hypg21\bigl(1,2-\nobreak r_2,r_1;1-1/u)$ is holomorphic on the
interval $u=(0,1)$, and the contributions from the integrals on both
sides cancel each other. The factor $(1/u)^{1-r_3}$ in the first term
of \eqref{KR1} is the same on both sides of the interval. The imaginary
part of $1-1/u$ at $u=x\pm i\e$ is $\frac{\pm i \e}{x^2+\e^2}$. This
means that the factor $(1-\nobreak1/u)^{1-r_1}$ should be computed as
$\bigl( \frac{1-u}u\bigr)^{1-r_1}\, e^{\pm \pi i (1-r_1)}$. The
integral in \eqref{ui} is equal to
\begin{align*}e^{-\pi i r_3} &\,(\tau_1-\bar\tau_1)^{r_3-1} \int_{u=0}^1
f_3[qu] \,
(1-qu)^{-r_3} \, \frac{\Gf(r_1-1)\,\Gf(r_3)}{\Gf(2-r_2)} \, u^{r_3-1}\\
&\qquad\hbox{} \cdot
 \bigl(1/u-1\bigr)^{1-r_1}\, \frac{du}u\, \Bigl( - e^{\pi i(1-r_1)} +
 e^{-\pi i(1-r_1)} \Bigr)
 \displaybreak[0]\\
 &\= 2\pi i\, e^{-\pi i r_3} \,(\tau_1-\bar\tau_1)^{r_3-1}
 \frac{\Gf(r_3)}{\Gf(2-r_1)\, \Gf(2-r_2)}\, \int_{u=0}^1 f_3[qu]\,
 (1-qu \bigr)^{-r_3} \\
 &\qquad\hbox{} \cdot
  u^{1-r_2}\, (1-u)^{1-r_1}\, du\,.
\end{align*}
This integral is holomorphic in $r_1$ and $r_2$ under the conditions
$r_1<2$, $0<r_2<2$, $r_3=4-r_1-r_2>0$. The integral in the left hand
side of the relation in the lemma is holomorphic in $r_1$ and $r_2$ as
well. So we can drop the assumptions that $r_1$ and $r_2$ are not
integral.

As $u$ runs from $0$ to $1$, the image
$z=\frac{\tau_1-\bar\tau_1 q u}{1-qu}$ runs from $\tau_1$ to
$\frac{\tau_1-\bar\tau_1 q}{1-q} = \tau_2$ along the geodesic segment
$s_{\tau_1,\tau_2}$ in~$\uhp$. Carrying out the backward substitutions
$u=w/q$ and $w=\frac{z-\tau_1}{z-\bar\tau_1}$ we arrive at
\bad 2\pi i&\, e^{-\pi i r_3}
(\tau_1-\bar\tau_1)^{r_3-1}\, \Bt(2-r_1,2-r_2)^{-1}
\int_{z=\tau_1}^{\tau_2} f_3(z)\, \Bigl(
\frac{z-\bar\tau_1}{\tau_1-\bar\tau_1} \Bigr)^{r_3}
\\
&\qquad\hbox{} \cdot
\Bigl(
\frac{(z-\tau_1)(\tau_2-\bar\tau_1)}{(z-\bar\tau_1)(\tau_2-\tau_1)}\Bigr)^{1-r_2}\,
\Bigl(
\frac{(\tau_2-z)(\tau_1-\bar\tau_1)}{(z-\bar\tau_1)(\tau_2-\tau_1)}
\Bigr)^{1-r_1}\\
&\qquad\hbox{} \cdot
\frac{(\tau_2-\bar\tau_1)(\tau_1-\bar\tau_1)}{(\tau_2-\tau_1)(z-\bar\tau_1)^2}\,
dz\,, \ead
with the beta-function $\Bt(x,y)=\Gf(x)\Gf(y)/\Gf(x+\nobreak y)$. The
path of integration is the geodesic segment ~$s_{\tau_1,\tau_2}$. To
handle the powers we note that $z-\bar\tau_1$ and $\tau_1-\bar\tau_1$
are both in the upper half-plane. So with the standard choice of the
argument in $(-\pi,\pi)$ we have
\[ (\tau_1-\bar\tau_1)^{r_3-1}\, \Bigl(
\frac{z-\bar\tau_1}{\tau_1-\bar\tau_1} \Bigr)^{r_3} =
(\tau-\bar\tau_1)^{-1} \, (z-\bar\tau_1)^{r_3}\,.\]
In
$\Bigl( \frac{(z-\tau_1)(\tau_2-\bar\tau_1)}{(z-\bar\tau_1)(\tau_2-\tau_1)}\Bigr)^{1-r_2}$,
coming from $u^{1-r_2}$, the arguments of the quotients
 $\frac{z-\tau_1}{\tau_2-\tau_1}$ and
 $\frac{z-\bar\tau_1}{\tau_2-\bar \tau_1}$ are equal, and contained in
 $(-\pi,\pi)$. So we can split up this power correspondingly. The factor
coming from $(1-\nobreak u)^{1-r_1}$ can be handled analogously. This
leads to the following:
\begin{align}\nonumber 2\pi i&\, e^{-\pi i r_3}\, \Bt(2-r_1,2-r_2)^{-1}
\int_{z=\tau_1}^{\tau_2} f_3(z)\, (z-\bar\tau_1)^{r_3-2} \, \Bigl(
\frac{z-\tau_1}{\tau_2-\tau_1} \Bigr)^{1-r_2}\, \Bigl(
\frac{z-\bar\tau_1}{\tau_2-\bar\tau_1}\Bigr)^{r_2-1}\\
\nonumber
&\qquad\hbox{} \cdot
\Bigl( \frac{\tau_2-z}{\tau_2-\tau_1}\Bigr)^{1-r_1}\, \Bigl(
\frac{z-\bar\tau_1}{\tau_1-\bar\tau_1}\Bigr)^{r_1-1}\,
\frac{\tau_2-\bar\tau_1}{\tau_2-\tau_1}\, dz
\displaybreak[0]\\
\nonumber
&\= -\pi \, 2^{2-r_1}\, e^{-\pi i (r_1/2+r_3)}\, \Bt(2-r_1,2-r_2)^{-1}\,
(\tau_2-\bar\tau_1)^{2-r_2}\, \bigl|c\tau_1+d\big|^{2-2r_1}\\
 &\qquad\hbox{} \cdot
\int_{z=\tau_1}^{\tau_2} f_3(z)\, \Bigl( \frac{z-\tau_1}{\tau_2-\tau_1}
\Bigr)^{1-r_2}\, \Bigl( \frac{\tau_2-z}{\tau_2-\tau_1}\Bigr)^{1-r_1}\,
\frac{dz}{\tau_2-\tau_1}\,.
\end{align}
In the second step we decomposed the powers of quotients in which
$\bar\tau_1$ occurs. This is possible since both numerator and
denominator are in the upper half-plane. We also used that
$\frac{di-b}{a-ic} = \tau_1$, hence
$\tau_1-\bar\tau_1= \frac {2i}{a^2+c^2}$, and $c\tau_1+d=\frac1{a+ic}$.

The resulting integrand is holomorphic in $z$, and we can deform the
path of integration. The two powers in the integrand are multi-valued
and have singularities at $\tau_1$ and $\tau_2$, respectively. In the
statement of the lemma we choose a region on which the integrand is
well-defined. If $\tau_2$ tends to $\tau_1$ the limit of the integral
exists.
\end{proof}

\begin{proof}[Completion of the proof of Theorem~\ref{thm-tri}] For
$\tau_1\neq\tau_2$ we obtain from Lemmas \ref{lem-ti1}
and~\ref{lem-ti2}:
\begin{align*}\bigl[f_3&, \kp(\tau_1,\tau_2;\cdot) \bigr]_{r_3} \=
-(2i)^{2-r_2} \, |c\tau_1+d|^{2r_1-2+2-2r_1}\,
(\tau_2-\bar\tau_1)^{r_2-2+2-r_2}
\\
&\qquad\hbox{} \cdot
2^{2-r_1}\, e^{-\pi i(r_1/2+r_3) } \, \Bt(2-r_1,2-r_2)^{-1}\\
&\qquad\hbox{} \cdot
\int_{z=\tau_1}^{\tau_2}\, f_3(z) \,\Bigl(
\frac{z-\tau_1}{\tau_2-\tau_1} \Bigr)^{1-r_2}\, \Bigl(
\frac{\tau_2-z)}{\tau_2-\tau_1}
\Bigr)^{1-r_1}\,\frac{dz}{\tau_2-\tau_1}
\displaybreak[0]\\
&\= (-2i)^{r_3}\, \Bt(2-r_1,2-r_2)^{-1}\\
&\qquad\hbox{} \cdot
\int_{z=\tau_1}^{\tau_2}\, f_3(z) \,\Bigl(
\frac{z-\tau_1}{\tau_2-\tau_1} \Bigr)^{1-r_2}\, \Bigl(
\frac{\tau_2-z)}{\tau_2-\tau_1}
\Bigr)^{1-r_1}\,\frac{dz}{\tau_2-\tau_1}
\displaybreak[0]\\
&\= (-2i)^{r_3}\, \Bt(2-r_1,2-r_2)^{-1}
\\&\qquad\hbox{} \cdot
\int_{u=0}^1 f_3\bigl( \tau_1+u(\tau_2-\tau_1)\bigr)\,
u^{1-r_2}\,(1-u)^{1-r_1}\, du\,.
\end{align*}
The quantities $\bar\tau_1$, $c$ and $d$, which are not holomorphic in
$\tau_1$ and $\tau_2$ cancel in this result. The second version shows
that the whole expression is holomorphic in
$(\tau_1,\tau_2) \in \uhp^2$.

These expressions can be inserted for
$\bigl[ f_3,\kp(\tau_1,\tau_2)\bigr]_{r_3}$ in
\[ \T_a(f_1,f_2,f_3) \= \int_{\tau_2=i}^{ia}
f_2(\tau_2)\int_{\tau_1=\rho-1}^\rho f_1(\tau)\, \bigl[
f_3,\kp(\tau_1,\tau_2)\bigr]_{r_3} \, d\tau_1\,d\tau_2 \,.\]

The values $\bigl[ f_3,\kp(\tau_1,\tau_2)\bigr]_{r_3}$ have at most
polynomial growth in $\tau_2$, since $f_3$ has at most polynomial
growth. The values of $f_1(\tau)$ stay bounded. The exponential decay
of the cusp form $f_2(\tau_2)$ as $\tau_2$ moves up to infinity, ensure
that the limit as $a\uparrow\infty$ exists. The limit is given by the
same expression with $ia$ replaced by~$\infty$. Lemma~\ref{lem-Tlimit}
implies that the resulting limit is equal to $\T(f_1,f_2,f_3)$.

The path of integration of the integral over $z$ can be deformed,
provided we take care not to cross singularities of the integrand. The
discontinuities of the powers of quotients in~\eqref{tii} can be chose
to occur along the geodesic half-lines indicated in the theorem.
\end{proof}

\subsection{Reformulations of the triple integral}
\subsubsection{Integration over a $3$-cycle} The triple integral
in~\eqref{tii} is based on the choice of the pairs
$(\rho-\nobreak1,\rho)$ and $(i,\infty)$ determining the limits of
integration in the outer integral.

Let us put
\be Y_3 \= \bigl\{ (\tau_1,\tau_2,z) \in \uhp^3\;:\; z\not\in
\ell_{\tau_1,\tau_2}\setminus s_{\tau_1,\tau_2} \bigr\}\,, \ee
with $\ell_{\tau_1,\tau_2}$ and $s_{\tau_1,\tau_2}$ as indicated in
Theorem~\ref{thm-tri}. On $Y_3$ we have the holomorphic $3$-form
\be \Om(\tau_1,\tau_2,z) \= \Bigl( \frac{z-\tau_1}{\tau_2-\tau_1}
\Bigr)^{1-r_2}\, \Bigl(\frac{\tau_2-z}{\tau_2-\tau_1} \Bigr)^{1-r_1}\,
\frac{dz\,d\tau_2\,d  \tau_1 }
  { e \tau_2  -\tau_1}\,. \ee
One can check that its transformation behavior is such that
\[ f_1(\tau-1)\, f_2(\tau_2)\, f_3(z) \,
  \Om(\tau_1,\tau_2,z) \]
is $\Gm$-invariant for the diagonal action of $\Gm$ on $\uhp^3$. Up to
the factor
\[ \frac{(-2i)^{r_3}}{\Bt(2-r_1,2-r_2)}\]
the triple integral $\T(f_1,f_2,f_3)$ is given by a specific
$3$-dimensional cycle in $Y_3$. Cutting up the cycle and using the
$\Gm$-invariance we can give alternative formulations of the triple
integral.

\subsubsection{Triple integral expressed in Fourier coefficients} We
turn to the question whether $\T(f_1,f_2,f_3)$ in Theorem~\ref{thm-tri}
might be zero for all choices of the modular forms~$f_j$. The
expression \eqref{tiu} for the triple integral is fairly explicit, but
for numerical purposes the following version is more convenient.
\begin{prop}\label{prop-triFc}Let $r_1,r_2,r_3,p_1,p_2,p_3\in \RR$
satisfy
\bad &r_1<2,\quad 0<r_2<2,\quad r_3> 0\, \quad r_1+r_2+r_3 \= 4\,,\\
&p_j \equiv r_j\bmod 2\,,\qquad p_1+p_2+p_3 \=0\,. \ead
We consider modular forms $f_1\in A_{r_1}\bigl( v[p_1]\bigr)$,
$f_2\in \cf{r_2}\bigl(v[p_2]\bigr)$,
$f_3\in M_{r_3}\bigl( v[p_3]\bigr)$ given by their Fourier expansions
\bad f_j(z) \= \sum_{m\geq 0} a_j(m) \, e^{2\pi i(m+p_j/12)z}\,. \ead
Then
\badl{ts} \T(f_1,f_2,f_3) &\= \frac{(-2i)^{r_3}}{\Bt(2-r_1,2-r_2)}
\sum_{m_1,m_2,m_3\geq 0} a_1(m_1)\, a_2(m_2)\, a_3(m_3) \\
&\quad\hbox{} \cdot
\Ps_{r_1,r_2}\bigl( \tfrac{12m_1+p_1}{12},\tfrac{12m_2+p_2}{12},
\tfrac{12m_3+p_3}{12}\bigr)\,, \eadl
 where we use
\begin{align} \nonumber
\Ps_{r_1,r_2}(\mu_1,\mu_2,\mu_3) & \= \frac{e^{-2\pi \mu_2-\pi\sqrt
3(\mu_1+\mu_3)}}{2\pi i} \int_{u=0}^1 u^{1-r_2}\, (1-u)^{1-r_1}\,
\frac{e^{-\pi(2-\sqrt 3)\mu_3 u}}{\mu_2+\mu_3 u}\\
\label{Psi-int1}
&\qquad\hbox{} \cdot
 S\bigl(\pi(\mu_1+(1-u)\mu_3)\bigr)\, du\\
 \nonumber
 &\= \frac{e^{-2\pi \mu_2-\pi\sqrt 3(\mu_1+\mu_3)}}{2\pi i}\,
 \mu_3^{1-r_3}\, \int_{u=0}^{\mu_3}\, u^{1-r_2}\, (\mu_3-u)^{1-r_1}\\
\label{Psi-int2}
 &\qquad\hbox{} \cdot
 \frac{e^{-\pi(2-\sqrt 3) u}}{\mu_2+u}\,
 S\bigl(\pi(\mu_1+\mu_3-u)\bigr)\, du \,,\\
 S(x) &\= \frac{\sin x}x\quad \text{ with smooth continuation at
 $x=0$}\,.
\end{align}
\end{prop}
\rmrk{Remark}Both versions of the integral for $\Ps_{r_1,r_2}$ are
trivially equivalent. Version \eqref{Psi-int2} seems slightly simpler
for numerical integration. The factor $S(\cdots)$ may oscillate,
depending on $\mu_1\in \RR$ and $\mu_3>0$.
\begin{proof}
The absolute convergence of \eqref{tiu} and the uniform absolute
convergence of the Fourier expansions of the modular forms $f_j$ in the
domains occurring in~\eqref{tiu} allow us to interchange the order of
integration and summation. We compute the resulting integral for a
triple of individual Fourier terms. Inserting it gives the triple
sum~\eqref{ts}.

Let $\mu_1\in \RR$, $\mu_2>0$, $\mu_3\geq 0$. We have to consider
\begin{align*}\int_{\tau_1=\rho-1}^{\rho} &e^{2\pi i \mu_1 \tau_1}
\int_{\tau_2=i}^\infty e^{2\pi i \mu_2 \tau_2} \int_{u=0}^1 e^{2\pi i
\mu_3(\tau_1+u(\tau_2-\tau_1)}\\
&\qquad\hbox{} \cdot
u^{1-r_2}\,
(1-u)^{1-r_1}\, du\, d\tau_2\, d\tau_1\,.
\end{align*}

The absolute convergence allows us to carry out the integration of
$\tau_2$ and $\tau_1$ first. That brings us to the integral
\bad
 &\= \frac{e^{-\pi(2\mu_2+\sqrt 3(\mu_1+\mu_3))}}{2\pi i} \int_{u=0}^1
 u^{1-r_2}\, (1-u)^{1-r_1} \frac{e^{-\pi(2-\sqrt 3 )\mu_3
u}}{\mu_2+u\mu_3}\\
&\qquad\hbox{} \cdot
\frac {\sin \pi(\mu_1+(1-u)\mu_3)}{\pi(\mu_1+(1-u)\mu_3)}\, du\,. \ead

Except for a change of variables, we do not see a way for further
analytic treatment of this integral.
\end{proof}

\subsection{Numerical approach}\label{sect-num}

\rmrk{Choice of modular forms}For a numerical computation we consider
$f_1=\eta^{2r_1}$, $f_2=\eta^{2r_2}$, and
$f_3=E_4\, \eta^{-2(r_1+r_2)}$, with $r_1<-r_2$ and $0<r_2<2$. In this
situation we can apply Proposition~\ref{prop-triFc}.

\rmrk{Approach to compute the triple integral} We used \texttt{GP/Pari}
\cite{Pari} for the computation.

The Fourier coefficients of $E_4$ are known in terms of divisor sums,
and the Fourier expansion of powers of the Dedekind eta-function has
the form
\be \eta^{2r}(z) \= \sum_{m\geq 0} p_m(r)\, e^{2\pi i (m+r/12)z}\,, \ee
with polynomials $p_m$ of degree~$m$ in $\QQ[r]$, which can be
symbolically computed.

For given $(r_1,r_2)$ we start with a computation of the three lists of
Fourier coefficients up to a given order, and store them for use later
on. Then we have to compute the terms in the triple series
in~\eqref{ts}. For the evaluation of $\Ph_{r_1,r_2}$ we use the routine
\texttt{intnum} of Pari. The integrand may have problematic behavior at
the end points $u=0$ and $u=1$. This can be indicated in the arguments
of \texttt{intnum}. We can prescribe a desired precision.

\rmrk{Actual computations}The Pari computations of $\Ph_{r_1,r_2}$ give
consistent results under increase of the precision. We give in
Table~\ref{tab-comp} a number of computations for $f_1,f_2,f_3$ as
indicated at the start of this subsection.
\begin{table}[htp]
\[\begin{array}{|r|rrrr|}\hline
r_1& r_2=.2& r_2=.6& r_2=1.3 &r_2=1.8\\ \hline
-.3&7.911485&&&
\\
-.7&3.983793&3.811007&&
\\
-1.1&2.530185&2.706137&5.777819&
\\
-1.5&1.784794&2.070295& 5.008281&
\\
-2.4&0.993019& 1.313467&3.934317&22.868919
\\ \hline
\end{array}
\]
\caption[]{Computations of $\T(f_1,f_2,f_2)$, without the factor in
front of the integral in \eqref{ts}, and the factor $(2\pi i)^{-1}$
in~\eqref{Psi-int2}. }\label{tab-comp}
\end{table}

These results give evidence that \emph{the triple integral is non-zero,
and the cup product $\cp(f_1,f_2)$ is non-trivial at least in some
cases.}

\section{Universal covering group and principal series
representation}\label{sect-ucgps} In the previous sections we work with
the discrete subgroup $\Gm=\SL_2(\ZZ) $ of $\SL_2(\RR)$, and its action
on spaces of holomorphic functions. Some of the results that we need to
prove can be more conveniently considered in terms of the universal
covering group of $\SL_2(\RR)$ and its principal series
representations. See the Appendix in \cite{BCD} for a further
discussion.

\subsection{The universal covering group}The universal covering group
$\tG$ of $G=\SL_2(\RR)$ can be obtained from the Iwasawa decomposition
$G=PK$ of $G=\SL_2(\RR)$, where $K=\mathrm{SO}(2)$ and $P$ consists of
all upper triangular matrices in $G$. This gives a description of the
analytic variety $G$ as the product of a simply connected space
$P\cong\uhp$, and the space $\mathrm{SO}(2) \cong S^{\!1}$. The circle
$S^{\!1}$ has as simply connected covering the line $\RR$. This gives
$\tG$ as the space $\uhp \times\RR$. The group operations of $G$ can be
lifted in a unique way to $\tG$. In this way we get a central extension
\[ 0 \rightarrow \ZZ \rightarrow \tG \rightarrow G \rightarrow 1\,.\]
The center $\tZ$ of $\tG$ is isomorphic to $\ZZ$ and covers the center
$\{1,-1\}$ of~$G$. In \cite[(A.2)]{BCD} a useful section
$g\mapsto \tilde g$ of the homomorphism $\tG\rightarrow G$ is
indicated. We note that it cannot be a group homomorphism.

The group $\Gm=\SL_2(\ZZ)$ is covered by a discrete subgroup $\tGm$ of
$\tG$, generated by two elements $ \tilde{t} =\widetilde{\matc1101}$
and $ \tilde{s}=\widetilde{\matr0{-1}10}$. The relations between
$\tilde t$ and $\tilde s$ are generated by
$\tilde{t}\tilde{s}^2=\tilde{s}^2\tilde{t} $ and
$ \tilde{t}\tilde{s}\tilde{t}\tilde{s}\tilde{t} =\tilde{s}$. All
multiplier systems $v[p]$ correspond to a character $\ch_p$ of $\tGm$.
See \cite[(A.10)]{BCD}.

Characters of $\tGm$ correspond to multiplier systems. See
\cite[(A.10)]{BCD}. Let $\rho$ be a representation of $\tG$ in some
vector space $V$. One says that this representation has a \emph{central
character} if the center acts by $ \tilde{s}^2 \mapsto e^{-\pi i q}$
for some $q\in \CC$. The character $\ch_p$ of $\tGm$ corresponding to
the multiplier system $v[p]$ also satisfies
$\ch_p:  \tilde{s}^2 \mapsto e^{-\pi i p}$. Then
$\ch_p^{-1} \otimes \rho$ is a representation of $\tGm$ that is trivial
on $\tZ$. Hence it induces a representation of $\Gm$. We will see in
Proposition~\ref{prop-ident} that the representations
$\udsv{v,r}{-\om}$ and $\dsv{v^{-1},r}\om$ are of this form. So it is
useful to understand some representations of~$\tGm$ that have a central
character.

\subsection{Principal series representation}\label{sect-ps}In \cite[\S
A.2]{BCD} \emph{principal series representations} of the universal
covering group $\tG$ are discussed. We use realizations depending on
parameters for $s,p\in \CC$.
\be\label{V} \V^\om(s,p) \;\subset\; \V^\infty(s,p) \;\subset\;
\V^{-\infty}(s,p)
\;\subset\; \V^{-\om}(s,p) \,.\ee
The space $\V^\om(s,p)$ consists of holomorphic functions on some
neighborhood of $\proj\RR$ in $\proj\CC$, or equivalently of the
real-analytic functions on $\proj\RR$, so that the space of the
representation $\V ^\om(s,p)$ is the space $C^\om(\proj\RR)$ of
real-analytic functions on $\proj\RR$, the space $\V^ \infty(s,p)$ is
$C^\infty(\proj\RR)$, the space $\V^{-\infty}(s,p)$ is the space of
distributions on $\proj\RR$, and the space $\V^{-\om}(s,p)$ is the
space of hyperfunctions. Hyperfunctions on $\proj\RR$ are represented
by holomorphic functions on $U \setminus \proj\RR$ where $U$ is some
neighborhood of $\proj\RR$ in $\proj\CC$. Two representatives $h_1$ on
$U_1\setminus \proj\RR$ and $h_2$ on $U_2\setminus \proj\RR$ are
equivalent if $f_1-f_2$ on
$\left( U_1\cap U_2 \right)\setminus \proj\RR$ is the restriction of a
holomorphic function on $U_1\cap U_2$. The natural inclusion
in~\eqref{V} sends $\ph \in \V^\om(s,p)$ to the hyperfunction with
representative $h$ that is zero on $\lhp$ and equal to $\ph$ on
$\uhp\cap \mathrm{dom}(\ph)$.

The action of $\tG$ is given by the same formulas in all these spaces.
It is indicated in \cite[(A.23)]{BCD}. For $g=\matc abcd$ with
$\arg(ci+d) \in (-\pi,\pi)$ it has the form
\badl{act-ucg} \ph|^\prs_{s,p}\tilde g\,(t) &\;:=\; (a-ic)^{p/2-s}\,
\Bigl( \frac{t-i}{t-g^{-1} i} \Bigr)^{s-p/2}\\
&\quad\hbox{} \cdot
(a+ic)^{-s-p/2} \, \Bigl( \frac{t+i}{t-g^{-1}(-i)} \Bigr)^{s+p/2}\,
\ph(g t)\,. \eadl
The powers of $\frac{\tau-x_1}{\tau-x_2}$ with $x_1$ and $x_2$ both in
$\uhp$ or both in $\lhp$ are holomorphic on $\proj\CC$ minus a path
from $x_1$ to $x_2$. They are determined by the choice of
$\arg\frac{\tau-x_1} {\tau-x_2}=0$ at $\tau=\infty$.

The central character is determined by
$ \tilde{s}^2   \mapsto e^{-\pi i p}$. Since each element of $\tG$ can
be written as the product of an integral power of $ \tilde{s}^2 $ and
an element $\tilde g$, we have a complete description of the action.

\subsubsection{Disk coordinates}\label{sect-dc}The spaces in \eqref{V}
have an alternative characterization in \emph{disk coordinates} in the
variable $w=\frac{z-i}{z+i}$ on $\proj\CC$. To $|w|=1$ corresponds the
real projective line $\proj\RR$, the upper half-plane is determined by
$|w|<1$, and the lower half-plane by $|w|>1$ (including $w=\infty$,
which corresponds to $z=-i$).

Elements of $\ph\in \V^\om(s,p)$ have a \emph{polar expansion}
\be\label{pr} \ph[w]\=\sum_{n\in \ZZ} c_n \, w^n\,,\ee
with $c_n = \oh \bigl( (1+\nobreak\e)^{-|n|}\bigr)$ for some $\e>0$. For
$\V^\infty(s,p)$ the condition is weaker:
$c_n = \oh \bigl( (1+|n|)^{-A}\bigr)$ for all $A\geq 0$. The
distributions in $\V^{-\infty}(s,p)$ have a polar expansion with
coefficients satisfying $d_n = \oh \bigl( (1+|n|)^B\bigr) $ for some
$B>0$. This gives a duality between $C^\infty(\proj\RR)$ and the space
$C^{-\infty}(\proj\RR)$ of distributions by the bilinear form
\be \label{dual-distr}\Bigl\langle \sum_n d_n w^n , \sum_n c_n w^n
\Bigr\rangle \= \sum_n d_n\, c_{-n}\,.\ee
In this way each distribution determines a continuous linear form on
$C^\infty(\proj\RR)$ for the natural topology given by all supremum
norms of $h\in \proj\RR$ of the derivatives
$\bigl (\frac1w\partial_w)^k h$ for $k\geq 0$.

A representative of a hyperfunction is given by convergent Laurent
series
\be
\begin{cases} \sum_{n\in \ZZ}c_n^+ w^n &\text{ for }1-\e<|w|<1\,,\\
-\sum_{n\in \ZZ} c_n^- \, w^n &\text{ for } 1<|w|<1+\e\,.
\end{cases}\ee
One can go over to the representative
\badl{repr+_}&\begin{cases}
\frac{c_0}2+ \sum_{n \geq 1}c_n \, w^n &\text{ for }|w|<1\,,
\\
- \frac{c_0}2
-\sum_{n \leq -1}c_n\, w^n &\text{ for } |w|>1\,,
\end{cases}\\
&c_n \= c_n^++ c_n^-\,. \eadl
This is the unique representative $f$ that is holomorphic on
$\uhp\cup\lhp$ and satisfies $f(i)+f(-i)=0$.

\subsubsection{Duality}Let $h$, respectively $f$, represent an element
of $\V^{-\om}(s,p)$, respectively $\V^\om(1-\nobreak s,-p)$. Denote
$\tilde h(w) = h\bigl( \frac{z-i}{z+i}\bigr)$ and
$\tilde f(w)=f\bigl( \frac{z-i}{z+i}\bigr)$ the corresponding functions
in disk coordinates. Following \cite[\S2.1]{BLZ} we may consider
\badl{du} \langle h,f \rangle &\= \frac1{2\pi i} \biggl( \int_{|w|=c} -
\int_{w=c^{-1}}\biggr) \, \tilde h(w) \, \tilde f(w)\, \frac {dw}w\\
&\= \frac1\pi \biggl( \int_{z\in C_+} - \int_{z\in C_-}\biggr)\, h(z) \,
f(z) \, \frac{dz}{z^2+1}\,. \eadl

The constant $c\in (0,1)$ is such that $\tilde f$ is holomorphic on the
region $c\leq |w|\leq c^{-1}$ and $\tilde h$ is holomorphic on the
regions $c\leq |w|<1$ and $1<|w|\leq c^{-1}$. This corresponds to
contours $C_\pm$ in $\hp^\pm$ such that $f$ is holomorphic on the
region in $\proj\CC$ between $C_+$ and $C_-$, the contours included,
and $h$ is holomorphic on this region with $\proj\RR$ excluded. In this
way the choice of the contour does not influence the value of
$\langle h,f\rangle$. Moreover, if we replace $h$ by another
representative of the same hyperfunction we get the same value, so we
have obtained a duality between $\V^{-\om}(s,p)$ and
$\V^\om(1-\nobreak s,-p)$. One may check that if we take
$h\in \V^{-\infty}(s,p)$ then we get the duality in \eqref{dual-distr}.
Comparison with \eqref{act-ucg} shows that for all
$\al \in \V^{-\om}(s,p)$, $f\in \V^\om(1-s,-p)$
\be\label{gen-duality} \bigl\langle \al|^\prs_{s,p} g ,f |^\prs_{1-s,-p}
g \bigr\rangle \= \langle \al,f\rangle\qquad\text{ for all }g\in
\tG\,.\ee
In Theorem~\ref{thm-dual} we used brackets $[\cdot,\cdot]_r$ for a
bilinear $\Gm$-invariant duality between $\udsv{v,r}{-\om}$ and
$\dsv{v^{-1},r}\om$. Here we use the brackets
$\langle\cdot,\cdot\rangle$ for a more general duality between
hyperfunctions and analytic vectors. It is invariant for all principal
series actions of $\tG$. In \S \ref{sect-dualps} we will derive the
specialized bilinear form $[\cdot,\cdot]_r$ from the general duality
$\langle\cdot,\cdot\rangle$.

\subsubsection{Identifications}One can relate the $\Gm$-modules
$\dsv{v^{-1},r}{\pm \om}$ and $\udsv{v,r}{\pm \om}$ to submodules of
principal series representations of~$\tG$.
\begin{defn}\label{pmD-odef}Let $r\in \RR$.
\begin{enumerate}
\item[a)] $\pD^{-\om}(r/2,r) \subset \V^{-\om}(r/2,r)$ consists of the
hyperfunctions with a representative that is holomorphic on $\uhp$ and
zero on $\lhp$.
\item[b)]$\mD^{-\om}(r/2,-r) \subset \V^{-\om}(r/2,-r)$ consists of the
hyperfunctions with a representative that is holomorphic on $\lhp$ and
zero on $\uhp$.
\end{enumerate}
\end{defn}
Application of \eqref{act-ucg} to these representatives gives functions
of the same type. So the subspaces $\pD^{-\om}(r/2,r)$ of
$\V^{-\om}(r/2,r)$ and $\mD^{-\om}(r/2,-r)$ of $\V^{-\om}(r/2,-2)$ are
in fact submodules for these choices of the parameters $(s,p)$ in the
principal series representations. To check it we use that
\be\label{G0def} G_0 \= \Bigl\{\matc abcd \in \SL_2(\RR)\;:\; \arg(ci+d)
\in (-\pi,\pi)
\Bigr\}\ee
is an open neighborhood of $1$ in $\SL_2(\RR)$, which is mapped by the
lift $g\mapsto \tilde g$ to an open neighborhood of $1$ in~$\tG$. Hence
representations of the connected Lie group $\tG$ are determined by
their behavior on the $\tilde g$ with $g\in G_0$, or even $g$ in a
small neighborhood of $1$ in $\SL_2(\RR)$.

\begin{defn}\label{def-pmDom}The submodules
$\pD^\om(r/2,r) \subset \V^{-\om}(r/2,r)$ and
$\mD^\om(r/2,-r) \subset \V^{-\om}(r/2,-r)$ are defined by the
condition that the representative $h$ in Definition~\ref{pmD-odef} has
the property that the restriction of $h$ to $\hp^\pm$ extends
holomorphically to a neighborhood of $\hp^\pm \cup \proj\RR$ in
$\proj\CC$. \end{defn}
Of course the extension of the restriction to the half-plane on which
$h$ is zero is trivial. It is an easy check with \eqref{act-ucg} that
this extension property is preserved by the action of $\tG$.

The six multiplier systems $v$ on~$\Gm$ for a given real weight $r$ are
$v[p]$ in \eqref{vpdef}, where $p\in \RR/12\ZZ$ satisfies
$p\equiv r\bmod 2$. These multiplier systems correspond to characters
$\ch$ of $\tGm$ by the relation $\ch(\tilde \gm) = v(\gm)$ for
$\gm\in \Gm$. The central character of $\ch$ is determined by
$\ch(s^2) = e^{-\pi i r}$, the same as for $\V(s,p)$. Hence the
representation $\ch^{-1}\otimes \V(s,r)$ is a representation of $\tGm$
that is trivial on the center $\tZ$ of $\tGm$. So it is in fact a
representation of $\tGm/\tZ \cong \PSL_2(\ZZ) \cong \Gm/\{1,-1\}$. We
can view it as a representation of $\Gm$ that is trivial on~$-1$.

\begin{prop}\label{prop-ident}
\begin{enumerate}
\item[i)]For a multiplier system of $\Gm$ for the real weight $r$,
corresponding to the character $\ch$ of~$\tGm$ the following
equivalences hold:
\bad \ch^{-1}\otimes \pD^{-\om}(r/2,r) &\; \cong\; \udsv{v,r}{-\om}\\
\ch \otimes \mD^{-\om}(r/2,-r)&\; \cong\; \dsv{v^{-1},r}{-\om} \ead
\item[ii)] Under these equivalences the submodules obtained by replacing
$-\om$ by $\om$ are equivalent as well.
\item[iii)] If $r\in \ZZ_{\leq 0}$ the submodule $\udsv{v,r}\pol$
corresponds to the submodule of $\ch^{-1} \otimes \mD^\pol(r/2,-r)$
spanned by the functions $w\mapsto w^{q}$ with integers $q$ satisfying
$0\leq q \leq |r|$.
\item[iv)] If $r\in \ZZ_{\leq 0}$ the submodule $\dsv{v^{-1},r}\pol$
corresponds to the submodule of $\ch \otimes \mD^\pol(r/2,-r)$ spanned
by the functions $w\mapsto w^{-q}$ with integers $q$ satisfying
$0\leq q \leq |r|$.
\end{enumerate}
\end{prop}
\begin{proof}Let $h^+$ be the holomorphic function on $\uhp$ that
together with $0$ on $\lhp$ represents a given element of
$\pD^{-\om}(r/2,r)$. The description in \eqref{act-ucg} implies for
$g=\matc abcd\in G_0$ that
\be h^+ |^\prs_{r/2,r} \tilde g (z) \= (a+ic)^{-r} \Bigl(
\frac{z+i}{z-g^{-1}(-i)} \Bigr) ^r \, h^+(g z)\,.\ee
Now put $h(z) = (-i-z)^{-r} \, h^+(z)$.
\be \label{h+a}(-i -z)^{-r} \, \bigl( h^+|^\prs_{ r/2, r }\tilde
g\bigr)(z)
\=
(cz+d)^{-r} \, h(g z)\,.\ee
To check this relation we first take $\matc abcd$ near to $\matc1001$.
Then we can take apart the powers of products and quotients. The
resulting formula extends to $g\in G_0$ by analyticity.

In a similar way $f(z) = (i-z)^{-r}f^-(z)$ on $\lhp$ leads to the other
isomorphism in part~i).

Part ii) is obtained by checking the definitions, where the holomorphy
at $\infty$ requires a bit of care.

The submodule $\dsv{v^{-1},r}\pol$ consists of the polynomial functions
of degree at most~$|r|$. Some computations on the basis of
$f^-(t)=(i-\nobreak t)^{|r|} \, f(t)$ and $w=\frac{t-i}{t+i}$ show that
polynomials in $t$ of degree at most $|r|$ correspond to polynomials in
$w^{-1}$ of degree at most~$|r|$.
\end{proof}

So the $\Gm$-modules that are of interest in this paper occur in
principal series representations of the universal covering group. We
note that we have used the functions $h^+$ and $f^-$ already in the
description in~\eqref{Dd}.

\subsubsection{Intertwining operators}\label{sect-intertw}In the
parameters $(s,p)$ of the principal series representations the
parameter $p$ is essentially determined modulo $2$. There is an
algebraic isomorphism
$\ell:\V^{-\om}(s,p) \rightarrow \V^{-\om}(s,p-\nobreak 2)$, given by
multiplication by $w=\frac{z-i}{z+i}$. It preserves the subspaces of
distribution vectors, smooth vectors and analytic vectors. See
\cite[(A.24, p.~157]{BCD}.

A more subtle relation exists between $\V^{-\om}(s,p)$ and
$\V^{-\om}(1-\nobreak s,p-\nobreak2)$. For general complex values of $s$ these
spaces are isomorphic. For the values of $s$ in which we are interested
the isomorphism breaks down, but there is a relation.

\begin{prop}\label{prop-ittwJ}
\begin{enumerate}
\item[i)] For each $r\in \RR$ here is an intertwining operator of
$\tG$-modules
\be\label{Jr} J_r : \V^\om(1-r/2,2-r) \rightarrow \V^\om(r/2,-r)\,\ee
given by the integral transformation
\be\label{Jrdef} J_r h (z) \= \frac1\pi \int_{\tau\in C} h(\tau)\,
\Bigl( \frac{2i(z-\tau)}{(\tau+i)(z-i)} \Bigr)^{-r}\,
\frac{d\tau}{(\tau+i)^2}\,, \ee
where the positively oriented cycle in $\uhp$ is inside the domain
of~$h$, encircles $-i$, and $z$ is outside~$C$.
\item[ii)] If $r\not\in \ZZ_{\leq 0}$ the image of $J_r$ is the module
$\mD^\om(r/2,-r)$, and the kernel is $\pD(1-r/2,2-r)$.
\item[iii)] If $r\in \ZZ_{\leq 0}$ then the image of $J_r$ is the module
$\mD^\pol(r/2,-r)$.
\end{enumerate}
\end{prop}
\begin{proof}Any $h \in \V^\om(1-\nobreak r/2,2-\nobreak r)$ has an
expansion $h = \sum_{m\in \ZZ} c_m\, w^m $ in the disk coordinate
$w=\frac{z-i}{z+i}$, where $c_m = \oh\bigl( (1+\e)^{-|m|}\bigr)$ for
some $\e>0$.

If we formulate the integral in \eqref{Jrdef} in disk coordinates $w$
and $v$, with the substitution $\tau= i \frac{1+v}{1-v}$, and insert
the expansion we arrive at
\be\label{Jr-w} J_r : \sum_{m\in \ZZ} c_m w^m \mapsto \sum_{m\leq -1}
c_m \frac{(r)_{-m-1}}{(-m-1)!}\, w^{m+1}\,.\ee

To show that $J_r$ is an intertwining operator we have to check that
\[ (J_r h)|^\prs_{r/2,-r} g = J_r\bigl( h|^\prs_{1-r/2,2-r} g\bigr)
\qquad \text{ for all }g\in \tG\,.\]
It suffices to do this for $\tilde g$ with $g$ in a small neighborhood
of $1$ in~$\SL_2(\RR)$, for which we do not have to worry about taking
apart powers of products and quotients. We leave this computation to
the reader.
(An alternative would be to look at the Lie algebra action on the weight
vectors $w^m$.) This completes the proof of part~i).

In part ii) we have $r\not\in \ZZ_{\leq 0}$ and the statements
concerning the kernel and the image of $J_r$ follow from~\eqref{Jr-w}.
For part iii)
we have $r=-a$ with $a\in \ZZ_{\geq 0}$. So $(r)_{-m-1}$ is nonzero if
and only $-m-1$ or $-a+|m|-2\leq -1$. This leads to possibly non-zero
terms in the image in \eqref{Jr-w} with $-a\leq m+1\leq 0$.
\end{proof}

A lift of the intertwining operator $J_r$ in~\eqref{Jr} is a linear map
$\s: \im J_r \rightarrow \V^\om(1-\nobreak r/2,2-\nobreak r)$ such that
$J_r\circ \s$ is the identity on $\im J_r$. A lift is in general not an
intertwining operator. With \eqref{Jr-w} it is not hard to describe a
lift.
\begin{prop}\label{prop-lift}\begin{enumerate}
\item[i)] For $r\in \RR \setminus \ZZ_{\leq 0}$ we define $\s_r \ph$ for
$\ph \in \mD^\om( r/2,-\nobreak r)$
\be \label{sgr-int}\s_{\!r}\ph(z) \= \frac1\pi \,\frac{z+i}{z-i}
\int_{\tau \in C} \ph(\tau) \hypg21
\Bigl(1,1;r;\frac{(\tau-i)(z+i)}{(\tau+i)(z-i)}\Bigr)\,
\frac{d\tau}{\tau^2+1}\,,\ee
where $C$ is a positively oriented closed curve in
$\proj\CC \setminus \{i,-i\}$ homotopic to $\proj\RR$ in the domain of
$\ph$, and where $z$ is outside~$C$.

The resulting linear map
$\s_r: \mD^\om(r/2,-r)\rightarrow \V^\om(1-\nobreak r/2,2-\nobreak r)$
is a lift of $J_r$.
\item[ii)] For $n\in \ZZ\leq 0$ we denote by $\ph_n$ the element of
$\mD^\om(r/2,-r)$ given by $\ph_n(w)=w^n$ (in disk coordinates). Then
\be \label{sr-bas}\s_r \ph_n \= \frac{|n|!}{(r)_{|n|}} \ph_{n-1}\ee
for all $r\in \CC\setminus \ZZ_{\leq 0}$.
\item[iii)] For $r\in \ZZ_{\leq 0}$ we define $\s_r$ on
$\mD^\pol(r/2,-r)$ by use of \eqref{sr-bas} for $r \leq n \leq 0$ on
the basis elements $\ph_n$. The resulting map
$s_r:\mD^\pol(r/2,-r) \rightarrow \V^\om(1-\nobreak r/2,2-\nobreak r)$
is a lift of $J_r$.
\end{enumerate}
\end{prop}
\begin{proof}Comparison with \eqref{Jr-w} shows that \eqref{sr-bas} in
part~ii) describes a lift of $J_r$ for $r\not\in \ZZ_{\leq0}$, and
gives part iii) as well.

Insertion of the power series of the hypergeometric function in
\eqref{sgr-int} and interchanging the order of summation and
integration leads to~\eqref{sr-bas}. This gives part~i).
\end{proof}

\rmrk{Freedom in the section}We have given one section $\s_r$ of $J_r$
with a simple description in the disk coordinate. We can add to $\s_r$
an arbitrary map from $\mD^\om(r/2,-r)$, respectively
$\mD^\pol(r/2,-r)$, to $\pD^\om(1-\nobreak r/2,2-\nobreak r)$ to obtain
another section.

We note that for all $g\in \tG$ and $\ph \in \mD^\om(r/2,-r)$
\be \label{trl-lift}\bigl( \s_r\ph)|^\prs_{1-r/2,2-r} g \quad\text{ and
}\quad \s_r\bigl( \ph|^\prs_{r/2,-r} g\bigr)\ee
have the same image under $J_r$. Hence their difference is in
$\pD^\om(1-\nobreak r/2,2-\nobreak r)$.

\subsection{Duality}\label{sect-dualps}
We formulate the duality theorem \ref{thm-dual} in the context of
suitable principal series representations.
\begin{thm}\label{thm-dual-ps}\begin{enumerate}
\item[i)] Let $r\in \RR \setminus \ZZ_{\leq 0}$. There is a
non-degenerate $\tG$-invariant bilinear form $[\cdot,\cdot]_r$ on
$\pD^{-\om}(r/2,r) \times \mD^\om(r/2,-r)$ given by
\be \label{di} [\al,f^-]_r \= \frac1\pi \int_{z\in C} h^+(z) \,\bigl(
\s_r f^- \bigr)(z)\,\frac{dz}{(z+i)^2}\,.\ee
Here $h^+$ is the holomorphic function on $\uhp$ that together with the
zero function on $\lhp$ represents the hyperfunction
$\al\in \pD^{-\om}(r/2,r)$, and $f^- \in \mD^\om(r/2,-r)$. The cycle
$C$ is homotopic to $\proj\RR$ in the intersection of $\uhp$ and the
domain of $\s_r f^-$.

The value of the bilinear form does not change if we add to $\s_r f^-$
any holomorphic function on $\uhp$.

\item[ii)] Let $r\in \ZZ_{\leq 0}$. The integral in \eqref{di}, now with
$f\in \mD^\pol(r/2,-r)$ defines a $\tG$-invariant bilinear form on
$\pD^{-\om}(r/2,r) \times \mD^\pol(r/2,-r)$. The bilinear form is
non-degenerate when restricted to
$\pD^\pol(r/2,r) \times \mD^\pol(r/2,-r)$.
\end{enumerate}\end{thm}

We first show how this theorem implies Theorem~\ref{thm-dual}, and next
prove the present version.
\begin{proof}[Proof of Theorem \ref{thm-dual}] Parts i) and iii) are
identical to part i) in Theorem~\ref{thm-dual} with use of the
identifications in Proposition~\ref{prop-ident}.

For part~ii) we note that if $h\in \udsv{v,r}{-\infty}$ then the
coefficients in the polar expansion \eqref{pr} of $h^+$ satisfy
$c_n \ll (1+n)^B$ for some $B>0$. The natural topology on
$\dsv{v^{-1},r}\infty$ given by the supremum norms of the derivatives
with respect to $-\cot t$ for $t\in \proj\RR$ can also be described by
the seminorms
$\sum_{m\leq 0} d_m w^m \mapsto \sup_{m\leq 0} |d_m| \, m^A$ for all
$A>0$. A computation shows that
\be \Bigl[\sum_{n\geq 0} c_n\, w^n, \sum_{m\leq 0} d_m \, w^{-m}
\Bigr]_r \= \sum_{n\geq 0} \frac {n!}{(r)_n}\, c_nd_m\,.\ee
This gives part~iv) of the theorem, and shows the linear form
$f\mapsto [h,f]_r$ extends continuously to $\dsv {v^{-1},r}\infty$.
\end{proof}

\begin{proof}[Proof of Theorem \ref{thm-dual-ps}] We start with the
duality in \eqref{gen-duality}, in the situation
\[ \langle\cdot,\cdot\rangle :\V^{-\om}(r/2,r-2) \times
\V^\om(1-r/2,2-r)\rightarrow \CC\,.\]

There is an algebraic isomorphism
$\ell:\V^{-\om}(s,p) \rightarrow \V^{-\om}(s,p-\nobreak 2)$, given by
multiplication by $w=\frac{z-i}{z+i}$. It preserves the subspaces of
distribution vectors, smooth vectors and analytic vectors. See the
characterization in terms of the polar expansion in \S\ref{sect-dc}.
Applying it to $\pD^{-\om}(r/2,r) \subset \V^{-\om}(r/2,r)$ we get a
submodule of of $\V^{-\om}(r/2,r-\nobreak 2)$. With use of the
expansion in the disk coordinate $w$ this submodule can be checked to
be orthogonal to $\pD^\om(1-\nobreak r/2,2-\nobreak r)$. So we get a
$\tG$-invariant duality
\[ \pD^{-\om}(r/2,r)\times \Bigl(\V^\om(1-r/2,2-r) \Bigm/
\pD^\om(1-r/2,2-r)\Bigr) \rightarrow\CC\,.\]

If $r\not\in \ZZ_{\leq 0}$ the section $\s_{\!r}$ of the intertwining
operator $J_r$ in Proposition~\ref{prop-lift} induces an isomorphism
\[ \mD^\om(r/2,-r) \cong \Bigl(\V^\om(1-r/2,2-r) \Bigm/
\pD^\om(1-r/2,2-r)\Bigr) \,.\]
Combining this we obtain for $\al \in \pD^{-\om}(r/2,r)$ and
$\ph \in \mD^\om(r/2,-r)$
\be\label{expl-du} \bigl[ \al,\ph ]_r \= \bigl\langle \ell \al, \s_{\!r}
\ph \bigr\rangle\,. \ee

Since $\al \in \pD^{-\om}(r/2,r)$, it has a representative given by a
holomorphic function $\ps$ on $\uhp$ and by $0$ on~$\lhp$. The
hyperfunction $\ell\al$ is represented by
$z\mapsto \frac{z-i}{z+i} \, h^+(z)$.

To see that the bilinear form is non-degenerate we note that in terms of
$w=\frac{z-i}{z+i}$ we have
\be \bigl\langle w^a, w^{-b}\bigr\rangle \= \frac{a!}{(r)_a}\,
\dt_{a,b}\,. \qedhere\ee

If $r\in \ZZ_{\leq 0}$, the lift $\s_r$ is defined only on
$\mD^\pol(r/2,-r)$. The bilinear form is defined on
$\pD^{-\om}(r/2,r)\times \mD^\pol(r/2,-r)$. In \eqref{expl-du} we see
that it is non-degenerate on $\pD^\pol(r/2,r) \times \mD^\pol(r/2,-r)$.
\end{proof}

\subsection{Variants of the integral in the duality
theorem}\label{sect-vidt} In the discussion in \S\ref{sect-trip} of the
triple integral we need an alternative way to describe the integrals in
the duality theorem. The proofs are easiest in the context of the
universal covering group~$\tG$.

The function
\be\label{kpdef} \kp(\tau_1,\tau_2;\tau) \= \Bigl(
\frac{\tau_1-\tau}{i-\tau}\Bigr)^{r_1-2}\,
\Bigl(\frac{\tau_2-\tau}{i-\tau}\Bigr)^{r_2-2}\ee
is holomorphic on $\uhp^3$ minus the set for which $\tau$ is on the
union of geodesic segment $s_{i,\tau_1}\cup s_{i,\tau_2}$. The function
 $\tau\mapsto \kp(\tau_1,\tau_2,\tau)$ extends holomorphically to
 $\proj\RR$ and $\lhp$, and determines an element of
 $\mD^\om(r_3/2,-r_3)$.

For functions holomorphic on the product of $\uhp^2$ and some
neighborhood of $\proj\RR$ in $\proj\CC$ we define two actions of
$\tilde G$:
\badl{tens-act} |_{r_1,r_2,r_3}^+ g &\= |_{2-r_1}g \otimes |_{2-r_2}g
\otimes |^\prs_{1-r_3/2,2-r_3}g\,,\\
 |_{r_1,r_2,r_3}^- g &\= |_{2-r_1}g \otimes |_{2-r_2}g \otimes
 |^\prs_{r_3/2,-r_3}g\,. \eadl
The action $|^\prs_{s,p}$ in the third factor of the tensor product is
given in~\eqref{act-ucg}. We define the other actions by taking for
 $M=\matc abcd $ near $1$ in $\SL_2(\RR)$
\[ h|_{2-r_1}\tilde M \otimes |_{2-r_2}\tilde M (\tau_1,\tau_2,z)
\= (c\tau_1+d)^{r_1-2}\, (c\tau_2+d)^{r_2-2}\, h(M\tau_1,M\tau_2,z)\,.\]
This determines an action of the universal covering group.

\begin{lem}\label{lem-kps3kp}Let $r_1,r_2\in \RR$, and put
$r_3=4-r_1-r_2$.
  \begin{enumerate}
  \item[i)] $\kp|_{2-r_1,2-r_2,r_3}^- g = \kp$ for all $g\in \tG$ and
  all $r_3\in \RR$.

\item[ii)] For given $\tau_1,\tau_2\in \uhp^2$,
$r_3\not\in \ZZ_{\leq 0}$ and $g\in \tG$
\[ (\s_{r_3}\kp) |_{2-r_1,2-r_2,r_3}^+ g\, (\tau_1,\tau_2,z)-
\s_{r_3}\kp
(\tau_1,\tau_2,z)\]
extends as a holomorphic function of $z$ on~$\uhp$.
\end{enumerate}
\end{lem}
\begin{proof} Part~i) is a consequence of the properties of Eichler
kernels. It can be checked for $g=\tilde M$ near to $1$ by a direct
computation. Since $\tG$ is connected this implies the invariance for
all~$g\in \tG$.

 For part~ii) we use Proposition \ref{prop-lift}. It gives the operator
 $\s_{r_3}$ as a lift of the intertwining operator $J_{r_3}$. Hence for
 for each $h\in \mD^\om(r_3/2,-r_3)$ and each $g\in \tG$
\begin{align*} J_{r_3}\Bigl( (\s_{r_3} h)|^\prs_{1-r_3/2,2-r_3} g\Bigr)
&\= \bigl( J_{r_3}\s_{r_3} h \bigr)\bigm|^\prs _{r_3/2,-r_3} g \= h
|^\prs _{r_3/2,-r_3} g\\
&\=J_{r_3}\Bigl( \s_{r_3} \bigl( h|_{r_3/2,-r_3}g\bigr)\Bigr) \,.
\end{align*}
So
$(\s_{r_3} h)|^\prs_{1-r_3/2,2-r_3} g -\s_{r_3} \bigl( h|_{r_3/2,-r_3}g\bigr)$
is in the kernel of $J_{r_3}$, hence holomorphic on~$\uhp$. The lift
$\s_{r_3}$ concerns only the third coordinate. Modulo functions of $z$
extending holomorphically to $\uhp$ we have
\bad (\s_{r_3}\kp)&|_{2-r_1}g\otimes |_{2-r_2} g\otimes
|^\prs_{1-r_3/2,1-r_3}g \= \s_{r_3} \bigl( \kp(|_{2-r_1}g\otimes
|_{2-r_2}g)\bigr) |^\prs_{1-r_3/2,2-r_3}g\\
&\;\equiv\; \s_{r_3}\bigl( \kp(|_{2-r_1}g\otimes g_{2-r_2}\otimes
g^\prs_{1-r_3/2,2-r_3}g)\bigr)
\= \s_{r_3}\kp\,, \ead
with use in the last step of the invariance of $\kp$ given in part~i).
\end{proof}

\begin{lem}\label{lem-skr}For $p\geq 1$,
$r_3=4-r_1-r_2 \not\in \ZZ_{\leq 0}$
\be \s_{r_3} \kp(i,ip;\cdot)(z)
\=2^{2-r_2}\,(p+1)^{r_2-2}\,\frac{z+i}{z-i}\, \hypg21\Bigl(
1,2-r_2;r_3;\frac{(p-1)(z+i)}{(p+1)(z-i)}\Bigr) \,,\ee
for $z$ outside a curve from $i$ to $ip$.
\end{lem}
\begin{proof}We use disk coordinates $w=\frac{\tau-i}{\tau+i}$,
$u=\frac{z-i}{z+i}$ and $v=\frac{ip-i}{ip+i}=\frac{p-1}{p+1}$. We have
$v\in (0,1)$ and should take $C_1$ corresponding to $|w|=c_1$ with
$v<c_1<1$. Then $\s_{r_3} \kp(i,ip;\cdot)(z)$ is applicable for
$|u|>c_1$.

We have
\begin{align*}\s_{r_3} &\kp(i,ip;\cdot)(z)
\= \frac1\pi \frac 1u \int_{|w|=c_1} 1^{r_1-2} \, \bigl(
\frac{1-v/w}{1-v}\Bigr)^{r_2-2}\,
\hypg21\bigl(1,1;r_3;w/u\bigr)\,\frac1{2i}\,\frac{dw}w
\displaybreak[0]\\
&\= \frac1{2\pi i}\,\frac 1u (1-v)^{2-r_2} \, \int_{|w_|=c_1}\sum_{n\geq
0} \frac{(2-r_2)_n}{n!} \frac{v^n}{w^n} \sum_{m\geq 0}
\frac{m!}{(r_3)_m} \, \frac{w^m}{u^m} \, \frac{dw}w
\displaybreak[0]\\
&\= u^{-1}(1-v)^{2-r_2} \hypg21\bigl(1,2-r_2;r_3;v/u\bigr)
\displaybreak[0]\\
&\= 2^{2-r_2}\,(p+1)^{r_2-2}\,\frac{z+i}{z-i}\, \hypg21\Bigl(
1,2-r_2;r_3;\frac{(p-1)(z+i)}{(p+1)(z-i)}\Bigr)\,.\qedhere
\end{align*}
\end{proof}

\begin{lem}\label{lem-gtf}Let $r_3=4-r_1-r_2\not\in \ZZ_{\leq 0}$, and
$\tau_1\neq \tau_2$ in $\uhp$. For $\al \in \pD^{-\om}(r_3/2,r_3)$
represented by the holomorphic function $h^+$ on $\uhp$ and zero on
$\lhp$, and $f^-(\tau)=\kp(\tau_1,\tau_2;\tau)$
\bad \bigl[ \al&, f^- \bigr]_{r_3} \= \frac{(2i)^{2-r_2}}\pi \,\bigl|
c\tau_1+d\bigr|^{2r_1-2}\, (\tau_2-\bar \tau_1)^{r_2-2} \int_{z\in C}
h^+(z) \\
&\qquad\hbox{} \cdot
 \bigl( \frac{z+i}{z-\bar \tau_1}\Bigr)^{1-r_3}\, \hypg21\Bigl(
1,2-r_2;r_3;\frac{\tau_2-\tau_1}{\tau_2-\bar\tau_1}\,
\frac{z-\bar\tau_1}{z-\tau_1}\Bigr)\,\frac{dz}{(z-\tau_1)(z+i)}\,, \ead
where $C$ is a wide closed curve in $\uhp$ encircling all singularities
of the integrand in $\uhp$, and where $g=\matc abcd\in \SL_2(\RR)$ and
$p$ satisfy $g\tau_1=i$, $g\tau_2=ip$, $p\geq 1$, and
$\arg(ci+\nobreak d) \in [0,\pi)$.
\end{lem}

\rmrk{Remarks}The matrix $g$ is unique for given different $\tau_1$ and
$\tau_2$.

This is a hybrid formula. The right hand side depends directly on
$\tau_1$ and $\tau_2$, and indirectly via $c$ and $d$. The integrand is
not holomorphic in $\tau_1$ and~$\tau_2$.

\begin{proof}In the integral for $[\al,f^-]_{r_3}$ in
Theorem~\ref{thm-dual-ps} we can replace
$\s_{r_3}\kp(\tau_1,\tau_2;\tau)$ by
$(\s_{r_3}\kp)|^+_{2-r_1,2-r_2,r_3} \tilde g \,(\tau_1,\tau_2;\tau)$,
by Lemma~\ref{lem-kps3kp}. The choice of $g$ is such that we can apply
the explicit formula in Lemma~\ref{lem-skr}. We try to write the result
in terms of $\tau_1=g^{-1}i$ and $\tau_2=g^{-1}(ip)$ as far as
possible. We use
\begin{align*}a+ic &\= \frac1{c\bar\tau_1+d}& a-ic&\=
\frac1{2\tau_1+d}\\
p+1&\= -i(g\tau_2-g\bar\tau_1)& p-1&\= -i(g\tau_2-g\tau_1)
\end{align*}
After some computations we arrive at the formula in the lemma. We handle
powers of products and quotients by first taking $\tau_1$ and $\tau_2$
near to each other, and hence $g$ near to the unit matrix. We observe
that both sides of the equality are real-analytic in $\tau_1$
and~$\tau_2$.
\end{proof}


\iflitnum
\newcommand\bibit[4]{
\bibitem {#1}#2: {\em #3;\/ } #4}
\newcommand\bibitq[4]{
\bibitem {#1}#2: {\em #3\/ } #4}
\else
\newcommand\bibit[4]{
\bibitem[#1] {#1}#2: {\em #3;\/ } #4.}
\newcommand\bibitq[4]{
\bibitem[#1] {#1}#2: {\em #3\/ } #4.}
\fi
\newcommand\bibitn[4]{} 
\raggedright

\setcounter{secnumdepth}{2} \tableofcontents

\end{document}